\documentclass[a4paper,12pt]{amsart}

\usepackage{amssymb,xspace,amscd,yhmath,mathdots,txfonts,
mathrsfs,textcomp, color}

\usepackage[matrix,arrow,curve]{xy}\CompileMatrices

\usepackage{hyperref}

\usepackage{yfonts}[1998/10/03]

\DeclareMathAlphabet{\mathpzc}{OT1}{pzc}{m}{it}

\numberwithin{equation}{section}

\theoremstyle{plain}

\newtheorem{thm}{Theorem}[section]

\newtheorem{lem}[thm]{Lemma}

\newtheorem{prop}[thm]{Proposition}

\newtheorem{esemp}[thm]{Example}

\theoremstyle{definition}

\newtheorem{defn}{Definition}[section]
\newtheorem{exam}[thm]{Example}

\newtheorem{rmk}[thm]{Remark}

\DeclareMathAlphabet{\mathpzc}{OT1}{pzc}{m}{it}

\DeclareMathOperator{\R}{{\mathbb{R}}}

\DeclareMathOperator{\Gf}{\mathbf{G}}

\DeclareMathOperator{\ad}{\mathrm{ad}}
\DeclareMathOperator{\Ad}{\mathrm{Ad}}

\DeclareMathOperator{\ct}{\mathrm{c}}

\DeclareMathOperator{\Z}{\mathbb{Z}}

\DeclareMathOperator{\at}{\mathfrak{a}}

\DeclareMathOperator{\im}{{\mathrm{Im}}}
\DeclareMathOperator{\re}{{\mathrm{Re}}}

\DeclareMathOperator{\C}{\mathbb{C}}

\newcommand\qt{\mathfrak{q}}
\newcommand\gt{\mathfrak{g}}

\newcommand\hg{\mathfrak{h}}

\newcommand\iq{\mathpzc{i}}

\newcommand\spt{\mathfrak{sp}}

\newcommand\Symm{\mathpzc{Symm}}

\newcommand\Jd{\mathrm{J}}

\newcommand\lgt{\mathpzc{L}}

\newcommand\mt{\textswab{m}}

\newcommand\SU{\mathbf{SU}}

\newcommand\Fi{\mathpzc{F}}

\newcommand\Mf{\mathpzc{M}}

\newcommand\gl{\mathfrak{gl}}

\newcommand\slt{\mathfrak{sl}}
\newcommand\su{\mathfrak{su}}

\newcommand\Pp{\mathfrak{P}}

\newcommand\Az{\mathpzc{A}}
\newcommand{\Bz}{\mathpzc{B}}

\newcommand\Go{\mathpzc{G}}
\newcommand{\lt}{\mathfrak{l}}
\newcommand\G{\mathfrak{G}}

\newcommand\isot{{\mathfrak{h}}}

\newcommand\vq{\mathpzc{v}}
\newcommand\wq{\mathpzc{w}}

\newcommand\Hf{\mathbf{H}}
\newcommand\Lcr{\mathscr{L}}

\newcommand\Der{\mathpzc{Der}}

\newcommand{\pnt}{{\mathrm{p}}}

\newcommand\hq{\mathpzc{h}}
\newcommand\xq{\mathfrak{x}}
\newcommand\yq{\mathfrak{y}}

   \def\DHLhksqrt#1#2{\setbox0=\hbox{$#1\sqrt{#2\,}$}\dimen0=\ht0
     \advance\dimen0-0.2\ht0
     \setbox2=\hbox{\vrule height\ht0 depth -\dimen0}
     {\box0\lower0.4pt\box2}}

\title[$CR$ Algebras]{On  transitive contact and $CR$ algebras}
\author[S.~Marini]{Stefano Marini}
\address{S.~Marini: 
Dipartimento di Matematica e Fisica, 
III Universit\`a di Roma, 
Largo San Leonardo Murialdo, 1
00146, Roma (Italy) }
\email{marinistefano86@gmail.com}
\author[C.~Medori]{Costantino Medori}
\address{C.\ Medori:
Dipartimento di Scienze Matematiche, Fisiche e Informatiche\\ 
Universit\`a di Parma\\ Parco Area delle Scienze 7/a (Campus), 43124 Parma
 (Italy)} \email{costantino.medori@unipr.it}

\author[M.~Nacinovich]{Mauro Nacinovich}
\address{M.Nacinovich:
Dipartimento di Matematica\\ II Universit\`a di Roma
``Tor Ver\-ga\-ta''\\ Via della Ricerca Scientifica\\ 00133 Roma
(Italy)}
\email{nacinovi@mat.uniroma2.it}

\author[A.~Spiro]{Andrea Spiro}
\address{A.~Spiro:
Scuola di Scienze e Tecnologie,
Universit\`a di Camerino,
Via Madonna delle Carceri,
62032 Camerino (Macerata),
ITALY
}
\email{andrea.spiro@unicam.it}

\date\today
\thanks{The authors were  partially supported by the Project MIUR ``Real
and Complex Manifolds: Geometry, Topology and  Harmonic Analysis'' and
by GNSAGA of INdAM.}


\begin{document}

\begin{abstract} We consider locally homogeneous $CR$ manifolds and
show that, under a condition only depending on their underlying
contact structure, 
 their $CR$ automorphisms form a finite dimensional Lie group. 
\end{abstract}

\maketitle
\tableofcontents
\centerline{\textsc{Introduction}}\par \smallskip
In the past years some of the Authors introduced and 
investigated the notion of $CR$ algebra (see \cite{AMN10b,MN05}) to describe the local 
structure of homogeneous $CR$ manifolds. The understanding of local models 
is important e.g. for applying the method of E. Cartan to describe the
differential invariants of $CR$ structures. 
A key point is to find under which conditions
the infinitesimal automorphisms of the structure form a finite dimensional Lie algebra.
A $CR$ structure can be defined by the datum of a smooth involutive \textit{complex} distribution.
The real parts of its vectors define a \textit{real} distribution. A strong version of the condition that
the $CR$ manifold is not foliated by $CR$ submanifolds of lower dimension is that this real distribution is
a (generalised) \textit{contact} distribution, i.e. that its iterated commutators span the full tangent space.
The strong interplay between $CR$  and underlying contact structures
 was clearly exploited  in the work of N.~Tanaka 
(see \cite{Tan67,Tan70}). He considered, at each point,
the 
nilpotent $\Z$-graded real Lie algebra $\mt$ 
canonically associated to
a contact structure. 
To describe the infinitesimal automorphisms, one needs to consider 
\textit{extensions}, or prolongations  $\G$ of $\mt.$ 
In this setting, they
can be described recursively in terms of derivations of $\mt.$
When a $CR$ structure is imposed on the contact distribuiton,
finite dimensionality of the maximal
prolongation is, for this $\Z$-graded model, equivalent to the fact that the vector valued Levi form has trivial kernel. 
Thus Cartan's method applies, via Tanaka's theory, to the case where the 
associated $\Z$-graded 
Levi-Tanaka algebras (cf. \cite{MN97})
are isomorphic at all points and the Levi form is nondegenerate. \par 
The idea of introducing $CR$ algebras in \cite{MN05} originated from the observation that many interesting
homogeneous examples of $CR$ manifolds lead to infinite dimensional Levi-Tanaka algebras, their Levi
forms having  nontrivial kernels.
 An obvious generalization of the nondegeneracy condition is to require
that the iterated Levi forms have a trivial kernel. This condition,
that was called \textit{weak nondegeneracy} 
in \cite{MN05}
and is equivalent to the notion of \textit{(Levi) $k$-nondegeneracy} used by other authors (see e.g. \cite{BER1999}), 
 is indeed equivalent to the fact that the corresponding
$CR$ manifold  is not the total space of a $CR$ fibration with complex fibres.
 The differential invariants for $CR$ manifolds of hypersurface type  in real dimension $5$
 satisfying the notion of $k$-nondegenericity
 have been so far  studied by several authors with different techniques
(see e.g. \cite{KaFe2008,MeSp2014, IsZa2013, Poc2013, MeSp2015}).
Further developments in higher dimensions appeared in 
\cite{Sa2015,Po2015}.
A theory of invariants for 2-nondegenerate $CR$ hypersurfaces in arbitrary dimension,
modeled on Tanaka's approach, has been recently 
developed by Porter and Zelenko 
in~\cite{PoZe2017}.
\par 
In this paper we address the question 
on finite dimensionality of the full group of the automorphisms 
of 
$CR$ manifolds 
of arbitrary $CR$ dimension
and $CR$ codimension,  
whose 
$CR$ structure is locally homogeneous. 
Weak
nondegeneracy is 
a much more restrictive condition 
that the one we found to guarantee 
the finite dimensionality of the Lie algebra of
infinitesimal $CR$ automorphisms. 
In fact, our criterion only involves the underlying contact structure.
We found this fact very interesting. 
Indeed, it is preliminary to an approach where 
this 
(generalized)
contact structure is a priori given
as a characteristic of the manifold on which the
addition of a $CR$ structure is meant to 
modelling different geometrical or physical situations.
Our condition 
was called \textit{ideal nondegeneracy} 
in \cite{MN05}, where the fact that it was a sufficient criterion
for the finite dimensionality of the maximal extension was
correctly stated;  however, 
in the proof given there there was a gap that we fill here in
\S\ref{finitecr}.
\par 
Our proof of the existence 
of maximal  extensions of $CR$ algebras relies 
on a review of the classical work
on transitive geometry (see e.g. 
\cite{Conn81,Conn84,Gu1968,GS}), 
allowing us to substitute formal power series
to the canonical construction of  Tanaka in the $\Z$-graded case. 
Our discussion is restrained mostly at
a purely algebraic level. Thus, for a better understanding of
the geometrical significance of our results, 
we refer the reader  to \cite{AMN2013}
for a thorough introduction to $CR$ and homogeneous $CR$ manifolds.
\par 
Let us briefly describe the contents of the paper. \par
\S\ref{secprel} collects some general notions we 
thought relevant for the exposition.
Contact and transitive pairs and triples and $CR$ algebras
are defined, 
not restraining to finite dimensionality. 
We explicitly required that
the (possibly infinite dimensional)
 Lie algebras involved
 have a topological structure, although this structure is implicitly defined by the requirement 
that the isotropy subalgebra is closed and 
has finite codimension. We also list various  
nondegeneracy conditions
that will be investigated in the later sections.\par 
In \S\ref{descat} we construct a canonical descending chain 
of subspaces which is associated to a contact pair 
 to explain \textit{contact nondegeneracy}.
\par
An analogous construction in  
\S\ref{desccr}, characteristic of $CR$ algebras, 
describes
weak $CR$-nondegeneracy. We show by an example that it  
is in fact
a much more restrictive condition 
than nondegeneracy of the underlying contact structure.\par
We found convenient to explain in 
\S\ref{hcs} 
the way the abstract
contact triples of 
\S\ref{secprel} 
relate
to actual homogeneous
contact manifolds, to motivate 
our 
later
use 
of transitive contact geometry. \par
In \S\ref{secgrad} we introduce graded Lie algebras and the finiteness criterion of 
Noburu Tanaka. 
\par 
By using Tanaka's criterion, we prove in \S\ref{finitecr}
our main result, which states that 
\textsl{$CR$ algebras whose corresponding contact
triple is nondegenerate are finite dimensional.} \par
In \S\ref{sectrans} we deal with the general 
construction of the representation of transitive contact pairs
by structures involving vector fields 
with formal power series coefficients. 
This is the main tool in the
\textit{transitive geometry} of \cite{GS}: 
in this purely algebraic setting a germ of homogeneous space is
substituted by a topological Lie algebra 
in which the isotropy subalgebra is closed and has finite codimension. 
This describes a situation in which the values  of the infinitesimal automorphisms of the structure
span the full tangent space at a point.\par 
In the final \S\ref{ext} we utilize transitive geometry 
to construct maximal extensions of $CR$ algebras.
Then the result of \S\ref{finitecr} yields
the theorem that 
locally homogeneous $CR$ manifolds
with a nondegenerate underlying contact structure
have a finite dimensional Lie algebra of
infinitesimal $CR$ automorphisms and hence, in particular, their
$CR$ automorphisms make a 
finite dimensional Lie group.
\section{Definitions and preliminaries}\label{secprel}
In this section we introduce some notions which are relevant for an \textit{infinitesimal} description
of homogeneous (generalised) contact manifolds and of various geometrical structures that can be
defined on them. We are particularly interested in partial complex structures, and  $CR$
algebras (see \cite{MN05}) fall in this realm. 
\par 
A \textit{topological Lie algebra} is a Lie algebra $\gt_0$ over a topological field $\Bbbk$ 
with a fixed structure of  topological
Hausdorff vector space for which the Lie product is continuous. We say that $\gt_0$ is \emph{linearly compact}
if the intersection of any family of  affine subspaces of $\gt_0$ having the finite intersection property has
a nonempty intersection. (For details, see e.g. \cite{Conn84}).\par 
In the following we assume that $\gt_0$ is real and denote by $\gt$ its complexification.
Conjugation in $\gt$ is always understood with respect
to the real form~$\gt_0.$ For a $\C$-linear subspace $\lgt$ of $\gt$ we set 
\begin{equation}
 \Tilde{\lgt}_0=\{\re(Z)\mid {Z}\in\lgt\},\qquad \breve{\lgt}_0=\lgt\cap\bar{\lgt}\cap\gt_0.
\end{equation}
\begin{defn} \label{def1.1} 
\begin{itemize}
 \item A \emph{contact pair} is the pair $(\gt_0,\lgt_0)$ consisting of a linearly compact 
 topological real Lie algebra $\gt_0$
 and a closed linear subspace $\lgt_0$ of $\gt_0$ having a finite dimensional complement in $\gt_0$
and spanning $\gt_0$ as a Lie algebra. 
\item A \emph{contact $\C$-pair} is the pair $(\gt_0,\lgt)$ consisting of a linearly compact 
 topological real Lie algebra $\gt_0$
 and a closed complex
 linear subspace $\lgt$ of $\gt$ such that $(\gt_0,\Tilde{\lgt}_0)$
  is a contact pair. 
 \item A \emph{transitive pair} $(\gt_0,\isot_0)$ consists of a 
 linearly compact 
 topological Lie algebra $\gt_0$ and a closed subalgebra
$\isot_0$ having finite codimension in $\gt_0$ and not containing nontrivial ideals of $\gt_0.$ 
 \item A \emph{contact triple} is a triple $(\gt_0,\isot_0,\lgt_0)$ such that $(\gt_0,\lgt_0)$ is a contact pair,
$(\gt_0,\isot_0)$ a transitive pair, $\isot_0\subseteq\lgt_0$ and $[\isot_0,\lgt_0]\subseteq\lgt_0.$
 \item A \emph{contact $\C$-triple} is a triple $(\gt_0,\isot_0,\lgt),$ such that
 $(\gt_0,\isot_0)$ is transitive, $(\gt_0,\lgt)$ is a contact $\C$-pair, $\isot_0\subset\lgt$ and
 $[\isot_0,\lgt]\subseteq\lgt.$ 
 \item A \emph{fundamental almost $CR$ pair} is a contact $\C$-pair $(\gt_0,\lgt)$ for which 
 $\breve{\lgt}_0$ is a Lie subalgebra of $\gt_0$ and $[\breve{\lgt}_0,\lgt]\subseteq\lgt.$
The contact triple $(\gt_0,\breve{\lgt}_0,\Tilde{\lgt}_0)$ is said to be \emph{associated} to $(\gt_0,\lgt).$
 \item A \emph{fundamental $CR$ algebra} is an almost $CR$  
 pair $(\gt_0,\qt)$ such that $\qt$ is a complex Lie subalgebra of $\gt.$
\end{itemize}
\end{defn}
\begin{rmk}
 In the definition of a $CR$ algebra $(\gt_0,\qt)$ 
 of \cite{MN05} it was not required that 
 $(\qt+\bar{\qt})\cap\gt_0$ generates $\gt_0$ as a Lie algebra.
 When this is not the case and $(\gt_0,\qt_0)$ is a Lie algebra associated to a homogeneous
 $CR$ manifold $M_0$, then the manifold $M_0$ can be described, at least locally, as the product
 $M'_0\times{N}_0$ of a 
 $CR$ manifolds $M_0'$ having the same $CR$ dimension of $M_0$
 and a totally real $N_0.$ For many purposes we could
reduce to $M_0'$, which is a homogeneous $CR$ manifold 
whose $CR$ algebra at a point  has the same $\qt,$ while $\gt_0$
is substituted by the span of $\Tilde{\qt}_0.$  
\end{rmk}
We will use the following notions.
\begin{defn}[Nondegenracy conditions] \label{def1.1a} \quad \par
We say that
a contact triple $(\gt_0,\isot_0,\lgt_0)$ is  
\begin{itemize}
\item\emph{strictly nondegenerate} if $\{X\in\lgt_0\mid [X,\lgt_0]\subseteq\lgt_0\}=\isot_0.$ 
\item\emph{nondegenerate} 
if any ideal of $\gt_0$ which is contained in 
$\lgt_0$ is already contained in $\isot_0.$ 
\end{itemize}
A fundamental almost $CR$-pair $(\gt_0,\lgt)$ is 
\begin{itemize}
 \item \emph{strictly nondegenerate} if $\{Z\in\bar{\lgt}\mid [Z,\lgt]\subset\lgt+\bar{\lgt}\}=\lgt\cap\bar{\lgt}.$ 
 \item \emph{weakly non-degenerate} if there is no almost $CR$ pair $(\gt_0,\lgt')$ with
 $\lgt\subsetneqq\lgt'\subseteq\lgt+\bar{\lgt}.$  
 \item \emph{contact nondegenerate} if the associated contact triple $(\gt_0,\breve{\lgt}_0,\Tilde{\lgt}_0)$
 is nondegenerate.
\end{itemize}
\end{defn}
\begin{rmk} For a 
fundamental almost $CR$ pair,
strict nondegeneracy is equivalent to strict nondegeneracy of
the associated contact triple and 
implies weak nondegeneracy, which in turn
implies contact nondegeneracy.
 \end{rmk}

\section{Descending chain of a contact pair}\label{descat}

Given a contact pair $(\gt_0,\lgt_0)$, we construct a descending chain of $\R$-linear subspaces of $\gt_0$ 
\begin{equation}\label{eq2.2a}
\cdots \supseteq\Go_{-h}\supseteq\Go_{1-h}\supset\cdots\Go_{-1}\supseteq\Go_0\supseteq\Go_1\supseteq
\cdots\Go_h\supseteq\Go_{h+1}\supseteq \cdots 
\end{equation}
defining by recurrence 
\begin{equation}\label{eq2.2f} 
\begin{cases}
 \Go_{-1}=\lgt_0,\\
 \Go_{h}=\Go_{h+1}+[\Go_{h+1},\Go_{-1}], & \text{if $h<-1,$}\\
 \Go_{h}=\{X\in\Go_{-1}\mid [X,\Go_{-1}]\subseteq\Go_{h-1}\} & \text{if $h\geq{0}.$}
\end{cases}
 \end{equation}
 Since, by assumption, $\lgt_0$ is a  subspace of finite codimension
 that generates $\gt_0$ as a Lie algebra,
 there is a nonnegative integer $\muup$ such that $\Go_{-\muup}= \gt_0.$ Indeed
 the ascending chain of subspaces 
\begin{equation*}
 \{0\}=(\Go_{-1}/\Go_{-1})\subseteq (\Go_{-2}/\Go_{-1})\subseteq\cdots (\Go_{-h}/\Go_{-1})\subseteq
 (\Go_{-h-1}/\Go_{-1})\subseteq \cdots 
\end{equation*}
of the finite dimensional vector space $(\gt_0/\Go_{-1})$ stabilizes
and, by their definition,  if $(\Go_{-h}/\Go_{-1})=
 (\Go_{-h-1}/\Go_{-1}),$ then $\Go_{-r}=\Go_{-h}$ for all $r>h.$
For a contact pair ${\bigcup}_h\Go_{-h}=\gt_0,$  and hence $\Go_{-h}=\Go_{-h-1}=\gt_0$ 
for some $h>0.$
\begin{defn}
 The smallest nonnegative integer $\muup$ for which 
 $\Go_{-\muup}=\gt_0$ is called the \emph{depth},
 or \emph{kind}, 
 of the contact pair $(\gt_0,\lgt_0).$
\end{defn}
\begin{prop}\label{prop1.2f}
 Let $(\gt_0,\lgt_0)$ be a contact pair, \eqref{eq2.2a} the associated descending chain.
 Set 
\begin{equation}\label{eq1.3g}
 \ct_0={\bigcap}_{h\in\Z}\Go_h.
\end{equation}
  Then:
\begin{enumerate}
\item All $\Go_h$ are closed subspaces of $\gt_0.$ 
\item $\ct_0$ is the largest ideal of $\gt_0$ contained in $\lgt_0.$ 
\item \eqref{eq2.2a} is a filtration of $\gt_0.$ 
\item For all $h\geq{0},$ $\Go_h$ is a Lie subalgebra of $\gt_0.$ 
\end{enumerate}
\end{prop}
%%%%%%%%%%%%
\begin{proof} (1) The closedness of $\Go_{-1}$ 
 was assumed in the definition of a contact pair. 
For $h\geq{0}$ the statement follows by recurrence, 
because each $\Go_h$ is an intersection of the inverse images of 
the closed subspace $\Go_{h-1}$ 
by the continuous linear maps $\Go_{-1}\ni{X}\to[X,Y]\in\gt_0,$ for $Y$ varying in $\Go_{-1}.$ \par 
For $h<-1$, the statement is true because all $\R$-linear subspaces $V$ with $\lgt_0\subseteq{V}\subseteq\gt_0$
are closed in $\gt_0,$ since $\lgt_0$ is closed and has finite codimension in $\gt_0.$ 
This completes the proof of (1).
\par
(2) and the fact that $\Go_0$ is a Lie subalgebra of
$\gt_0$ are streighforward consequences of the defintions. \par 
To complete the proof, it suffices to check that \eqref{eq2.2a} if a filtration. 
We begin by checking the commutators of elements belonging  to  subspaces with negative indices. 
If $h<-1,$ and we assume that $[\Go_0,\Go_{h+1}]\subseteq\Go_{h+1}$, 
then 
\begin{align*}
 [\Go_0,\Go_h]&=[\Go_0,\Go_{h+1}+[\Go_{h+1},\Go_{-1}]]\\
 & \subseteq[\Go_0,\Go_{h+1}]+[[\Go_0,\Go_{h+1}],\Go_{-1}]
 +[\Go_{h+1},[\Go_0,\Go_{-1}]] \\
 &\subseteq
 \Go_{h+1}+[\Go_{h+1},\Go_{-1}]=\Go_h.
\end{align*}
This implies by recurrence that $[\Go_0,\Go_h]\subseteq\Go_h$ for all $h\leq{0}.$ \par 
Let now $h>0$ and assume that $[\Go_0,\Go_{h-1}]\subset\Go_{h-1}.$ 
By \eqref{eq2.2f} we already have $[\Go_h,\Go_{-1}]\subseteq\Go_{h-1}.$ Then
\begin{align*}
 [[\Go_0,\Go_h],\Go_{-1}]&\subseteq [[\Go_0,\Go_{-1}],\Go_h]+[\Go_0,[\Go_h,\Go_{-1}]]\\
 &\subseteq [\Go_h,\Go_{-1}]+[\Go_0,\Go_{h-1}]\subseteq \Go_{h-1}
\end{align*}
shows, by recurrence, that $[\Go_0,\Go_h]\subseteq\Go_h$ for all $h\geq{0.}$ \par 
By \eqref{eq2.2f} we have $[\Go_h,\Go_{-1}]\subseteq\Go_{h-1}$ for all
integers $h$ and this implies that $[\Go_{h_1},\Go_{h_2}]\subseteq\Go_{h_1+h_2}$ when either $h_1\leq{0}$
or $h_2\leq{0}.$ When both $h_1,h_2>0$, we can argue by recurrence on $h_1+h_2.$ In fact 
\begin{align*}
 [[\Go_{h_1},\Go_{h_2}],\Go_{-1}]&\subseteq  [[\Go_{h_1},\Go_{-1}],\Go_{h_2}]+ [\Go_{h_1},[\Go_{h_2},\Go_{-1}]]\\
& \subseteq  [\Go_{h_1-1},\Go_{h_2}]] [\Go_{h_1},\Go_{h_2-1}]\subseteq \Go_{h_1+h_2-1}
\end{align*}
if we assumed that $[\Go_{h'},\Go_{h''}]\subseteq\Go_{h'+h''}$ when $h'+h''<h_1+h_2.$ 
 This completes the proof of the fact that \eqref{eq2.2a} is a filtration
 and hence of the Proposition. 
\end{proof}
%\begin{defn} The smallest integer $\muup$ in (1) of Proposition~\ref{prop1.2f}
% is called the \emph{kind} of $(\gt_0,\lgt_0).$ 
%\end{defn}

\begin{lem}
 If $\isot_0$ is any Lie subalgebra of $\gt_0$ such that $[\isot_0,\lgt_0]\subseteq\lgt_0$, then 
\begin{equation}
\vspace{-21pt}
[\isot_0,\Go_h]\subseteq\Go_h,\;\;  \forall{h}\in\Z.\end{equation}
\qed \end{lem} 
 In particular, if $(\gt_0,\isot_0,\lgt_0)$ is a contact triple,
 then all subspaces $\Go_h$ of  the canonical filtration \eqref{eq2.2a} 
 are $\isot_0$-modules. 
%%%%%%
\begin{lem} Let \eqref{eq2.2a} be the canonical filtration of the contact pair of a contact triple
$(\gt_0,\isot_0,\lgt_0).$ Then, 
\begin{enumerate}
\item $(\gt_0,\isot_0,\lgt_0)$ is strictly nondegenerate if and only if $\isot_0=\Go_0.$
\item 
$(\gt_0,\isot_0,\lgt_0)$ is nondegenerate, if and only if there is a positive integer $k$ 
such that $\Go_k\subseteq\isot_0.$ \end{enumerate}
\end{lem} 
\begin{proof} Statement $(1)$ follows immediately from the definitions. \par 
To prove $(2),$ we note that,
with $\ct_0={\bigcap}_{h\in\Z}\Go_h$ as in \eqref{eq1.3g},
the nondegeneracy condition can be restated by saying that
$\ct_0\subseteq\isot_0.$ This is equivalent to the fact that the intersection of all subspaces 
$(\Go_h+\isot_0)/\isot_0$ in $\gt_0/\isot_0$ is $\{0\}.$ The statement follows 
because these subspaces
form a descending chain of vector subspaces of the finite dimensional vector space $\gt_0/\isot_0.$ 
\end{proof}
\begin{defn} The \emph{order of degeneracy} of a 
 contact triple $(\gt_0,\isot_0,\lgt_0)$ is the smallest positive integer $k$ 
 for which $\Go_{k}\subseteq\isot_0.$
\end{defn}
We observe that $0$ is \textit{strict} nondegeneracy, and \textit{degenerate} corresponds to $\infty$-degenerate.
\begin{exam}
 To a contact pair $(\gt_0,\lgt_0)$ we can always associate the  triples $(\gt_0,\{0\},\lgt_0)$
 and  $(\gt_0/\at_0,\Go_0/\at_0,\lgt_0/\at_0),$ 
 where $\at_0$ is the largest ideal of $\gt_0$ which is contained in $\Go_0.$ 
 The first one is a contact triple iff $\gt_0$ is finite dimensional. 
The second one is a contact triple provided $\gt_0/\Go_0$ is finite dimensional.
\end{exam}
\begin{rmk} In \S\ref{hcs} we will explain how 
a contact triple is canonically associated to
a homogeneous contact manifold, providing in this way a geometrical motivation for
Definition~\ref{def1.1}.
\end{rmk}
It is useful to reformulate the nondegeneracy conditions of \S\ref{secprel}
in terms
of iterated Lie brackets. We define by recurrence \begin{equation*}
\begin{cases}
[X_1]=X_1,\\
[X_1,X_2]=\text{Lie product of $X_1$ and $X_2$ in $\gt$},\\
[X_1,\hdots,X_k,X_{k+1}]=[[X_1,\hdots,X_k],X_{k+1}] ,
\;\;\text{for $k\geq{2}$},\\
\text{if}\;\; X_1,\hdots,X_{k+1}\in\gt,
\end{cases}
\end{equation*}
\begin{prop} A necessary and sufficient condition in order that
a contact triple $(\gt_0,\isot_0,\lgt_0)$ be $k$-nondegenerate is that
\begin{equation}
\forall X\in\lgt_0\setminus\isot_0,\; \exists\, X_0,\hdots,X_k\in\lgt_0
\;\;\text{such that}\;\; \, [X,X_0,\hdots,X_k]\notin\lgt_0.
\end{equation}
\par 
A necessary and sufficient condition in order that a $CR$ algebra
$(\gt_0,\qt)$ be weakly nondegenerate is that \begin{equation}
\vspace{-20pt}
\forall Z\in{\qt},\; \exists\, Z_0,\hdots,Z_k\in\qt
\;\text{such that}\;\, [\bar{Z},Z_0,\hdots,Z_k]\notin
(\qt+\bar{\qt}).
\end{equation} \qed
\end{prop}
\section{Descending chain of a  $CR$ algebra} \label{desccr}
Let $(\gt_0,\qt)$ be a $CR$ algebra. We already noted that 
strict nondegeneracy implies weak nondegeneracy. It is well known that  
the two conditions are not equivalent. 
Set $\Tilde{\qt}=(\qt+\bar{\qt})$ and $\Tilde{\qt}_0=\tilde{\qt}\cap\gt_0.$ These are not, in general, subalgebras, but 
only linear subspaces of $\gt$ and $\gt_0,$ respectively. 
To better understand weak $CR$ nondegeneracy, 
we construct recursively descending chains of complex Lie subalgebras and of complex vector subspaces 
of $\gt:$ 
\begin{gather} \label{eq1.1}
 \bar{\qt}^{(0)}\supseteq \bar{\qt}^{(1)}\supseteq \cdots\supseteq \bar{\qt}^{(h)}\supseteq \bar{\qt}^{h+1}\supseteq \cdots \\
 \label{eq1.2}
 \Tilde{\qt}^{(0)}\supseteq \Tilde{\qt}^{(1)}\supseteq \cdots\supseteq \Tilde{\qt}^{(h)}\supseteq \Tilde{\qt}^{h+1}\supseteq \cdots
\end{gather}
by setting 
\begin{equation} \label{eq2.3h} 
\begin{cases}
 \left. 
\begin{aligned}
& \bar{\qt}^{(0)}=\bar{\qt},\\
&\Tilde{\qt}^{(0)}=\qt+\bar{\qt}, 
\end{aligned}\right. %\} 
\qquad \text{for $h=0$,} \\[10pt]
\left. 
\begin{aligned}
 & \bar{\qt}^{(h)}=\{Z\in\bar{\qt}^{(h-1)}\mid [Z,\qt]\subseteq\Tilde{\qt}^{(h-1)}\},\\
 &\Tilde{\qt}^{(h)}=\qt+\bar{\qt}^{(h)}, 
\end{aligned}\right. %\} 
\qquad \text{for $h\geq{1}.$}
\end{cases}
 \end{equation}
%\begin{align} \label{eq2.3h}
%&
%\begin{cases} \begin{aligned}
% &\bar{\qt}^{(0)}=\bar{\qt},\\
% &\Tilde{\qt}^{(0)}=\Tilde{\qt}=\qt+\bar{\qt},\end{aligned}
% & \text{for $h=0} \\
% \left.\begin{aligned}
% &\bar{\qt}^{(h+1)}=\{Z\in\bar{\qt}^{(h)}\mid [Z,{\qt}]
% \subseteq \Tilde{\qt}^{(h)}\}, %& h\geq{0},
% \\
% &\Tilde{\qt}^{(h+1)}=\qt+\bar{\qt}^{(h+1)}, \end{aligned}\right\}
% & h\geq{0}.
%\end{cases} 
%\end{align}
\begin{lem}\label{lem1.1}
 We have
% \footnote{Define by recurrence 
% $[Z,W_1,W_2]=[[Z,W_1],W_2]$ and \par
% \qquad $[Z,W_1,\hdots,W_{m-1},W_m]
% =[[Z,W_1,\hdots,W_{m-1}],W_m]$ if $m\geq{3}.$}
 \begin{equation}\label{eqq1.4}
 {\bigcap}_{h\geq{0}}\bar{\qt}^{(h)}=
 {\bigcap}_{h\geq{1}}\{Z\in\bar{\qt}\mid [Z,W_1,\hdots,W_h]\in
\Tilde{\qt},\; 
\forall W_1,\hdots,W_h\in\qt\}.
\end{equation}
\end{lem} 
\begin{proof} Let us denote by $\Az,$  $\Bz$ 
the left and right hand side of \eqref{eqq1.4}, respectively.
Since
 $[\bar{\qt}^{(h)},\qt]\subset\qt+\bar{\qt}^{(h-1)}$ 
 for all  $h>0,$
% (by \eqref{eq2.3h}) 
we obtain that $\Az\subseteq{\Bz}.$ \par
 Vice versa, if $Z\in\bar{\qt}$ does not belong to $\Az,$
 then
  there is a positive integer $h$ 
  with $Z\notin\bar{\qt}^{(h)}.$ This means that 
  there is $W_1\in\qt$
such that $[Z,W_1]\notin\qt+\bar{\qt}^{(h-1)}.$ 
If $h=1,$ this suffices to show that $Z\notin{\Bz}.$ 
 If $h>1,$ 
write $[Z,W_1]=Z_1+W'_1$ with $Z_1\in\bar{\qt}$
and $W_1\in\qt.$ Since, by assumption,
 $Z_1\notin\bar{\qt}^{(h-1)}$ 
we can find $W_2\in\qt$ such that
$[Z_1,W_2],$ and hence also $[Z,W_1,W_2],$ 
does not belong to $\qt+\bar{\qt}^{(h-2)}.$ 
Iterating this argument, 
we show that there are $W_1,\hdots,W_h\in\qt$ such that 
$[Z,W_1,\hdots,W_h]\notin\qt+\bar{\qt}$ and therefore $Z$ does not belong to $\Bz.$
This completes the proof.
\end{proof}
\begin{lem}\label{lem1.2}
 All $\bar{\qt}^{(h)}$ are complex Lie algebras and  
\begin{equation}\label{qu'}
 \qt'= \qt + {\bigcap}_{h\geq{0}}\bar{\gt}^{(h)}={\bigcap}_{h\geq{0}}\Tilde{\qt}^{(h)}
\end{equation}
is the largest complex Lie subalgebra satisfying 
\begin{equation}
 \qt\subseteq\qt'\subseteq \qt+\bar{\qt}.
\end{equation}
\end{lem} 
\begin{proof} Let us show first 
that the $\bar{\qt}^{(h)}$'s are Lie
subalgebras. This holds true for $h=0,$
because the  conjugate $\bar{\qt}$ of $\qt$ with respect to the real
form $\gt_0$ is a complex Lie subalgebra of $\gt.$  
 Assume that we already know that $\bar{\qt}^{(h)}$ is a Lie algebra for some 
 $h\geq{0}.$ If $Z_1,Z_2\in\bar{\qt}^{(h+1)},$ we have $[Z_1,Z_2]\in\bar{\qt}^{(h)}$
 because $\bar{\qt}^{(h+1)}\subseteq\bar{\qt}^{(h)}$ and, by our inductive assumption, 
 $\bar{\qt}^{(h)}$ is a Lie algebra. Moreover, 
\begin{align*}
 [[Z_1,Z_2],{\qt}] & \subseteq [Z_1,[Z_2,\qt]]+[Z_2,[Z_1,\qt]]\subseteq
 [Z_1,\bar{\qt}^{(h)}+{\qt}]+[Z_2,\bar{\qt}^{(h)}+{\qt}]\\
 & \subseteq  [Z_1,\bar{\qt}^{(h)}]
 +[Z_1,\qt]+[Z_2,\bar{\qt}^{(h)}]+[Z_2,\qt]\subseteq 
 \bar{\qt}^{(h)}+{\qt},
\end{align*}
shows that $[Z_1,Z_2]\in\bar{\qt}^{(h+1)}.$ Clearly the 
right hand side of \eqref{qu'} 
is a Lie subalgebra of $\gt$ with
 $\qt\subseteq\qt'\subseteq\qt+\bar{\qt}.$ 
By Lemma~\ref{lem1.1} it is the maximal complex Lie subalgebra 
of $\gt$ containing $\qt$ and contained in $\qt+\bar{\qt}.$ 
In fact it
contains $\lt\cap\bar{\qt}$ for all complex Lie subalgebras 
with $\qt\subseteq\lt\subseteq\qt+\bar{\qt}.$ 
\end{proof}
\begin{prop}
The following are equivalent 
\begin{enumerate}
 \item  $(\gt_0,\qt)$ is weakly $CR$ nondegenerate;
 \item   ${\bigcap}_{h\geq{0}}\bar{\qt}^{(h)}=\qt\cap\bar{\qt};$ 
 \item $\bigcap_{h\geq{0}}\Tilde{\qt}^{(h)}={\qt}.$ \qed
\end{enumerate}
\end{prop}
Sequences \eqref{eq1.1} and \eqref{eq1.2} can be used 
to measure weak nondegeneracy. 
Let the \emph{lenght} of a descending chain 
of vector spaces 
\begin{equation*}
 V_0\supseteq V_1\supseteq \cdots\supseteq V_h\supseteq V_{h+1}
 \supseteq \cdots 
\end{equation*}
be
the smallest integer $\nuup$ such that $V_h=V_{\nuup}$ 
for all $h\geq\nuup.$ 
By Lemma~\ref{lem1.2} we have the statement:
\begin{prop} 
 The sequences \eqref{eq1.1} and \eqref{eq1.2} 
 have the same lenght $\nuup$ and all their terms with indices
 smaller than $\nuup$ are different. 
 \qed
\end{prop} 
\begin{defn}
 We say that $(\gt_0,\qt)$ is $(\nuup-1)$-nondegenerate 
 if the descending chains \eqref{eq1.1} and \eqref{eq1.2}
 have finite lenght $\nuup.$ 
\end{defn}
\begin{rmk}
 Strict nondegeneracy is $0$-nondegeneracy, while weak nondegeneracy is $\nuup$-nondegeneracy for some
 $\nuup<+\infty .$ 
\end{rmk}
\begin{exam} For $CR$ algebras,
 contact  is a weaker notion than weak nondegeneracy. A score of examples can be obtained 
 by considering real orbits in complex flag manifolds
 ( see \cite{AMN06, AMN06b}) whose $CR$ algebras are fundamental, but not weakly nondegenerate. 
 We give here a simple example, consisting of the minimal orbit $M_0$ of $\SU(1,5)$ in the complex
 flag manifold $M$ consisting of triples $(\ell_2,\ell_3,\ell_4)$ of complex $2,3,4$-planes in $\C^6$ with
 $\ell_2\subset\ell_3\subset\ell_4.$ A point of $M_0$ is a flag in the 
  $\SU(1,5)$-orthogonal of an isotropic line. Let us give the explicit matrix representation. We define  
  $\gt_0\simeq\su(1,5)$ and $\qt,\qt'\subset\slt_6(\C)$ by 
\begin{align*}
 \gt_0&=\left.\left\{ 
\begin{pmatrix}
 \lambdaup & \zetaup_1 & \zetaup_2 & \zetaup_3 & \zetaup_4 & \iq\sigma \\
 z_1 & \iq{t}_1 & -\bar{w}_1 & -\bar{w}_2 & -\bar{w}_3 & -\bar{\zetaup}_1 \\
 z_2 & w_1 & \iq{t}_2 & -\bar{w}_4 & -\bar{w}_5 & -\bar{\zetaup}_2\\
 z_3 & w_2 & w_4 &\iq{t_3}&-\bar{w}_6 & -\bar{\zetaup}_3\\
 z_4 & w_3 & w_5 & w_6 &\iq{t}_4 & -\bar{\zetaup}_4 \\
 \iq{s} & -\bar{z}_1 &-\bar{z}_2&-\bar{z}_3&-\bar{z}_4 & -\bar{\lambdaup}
\end{pmatrix}\right| \begin{gathered}
\lambdaup,z_i,\zetaup_i,w_i\in\C,\; s,\sigmaup,t_i\in\R,\\
2\im(\lambdaup)+{\sum}t_i 
=0\end{gathered}\right\},\\
\qt&=\left\{ 
\begin{pmatrix}
 * & * & * & * & * & *\\
  * & * & * & * & * & *\\
   0 & 0 & * & * & * & *\\
    0 & 0 & 0 & * & * & *\\
   0& 0 & 0 & 0 & * & *\\
   0& 0 & 0 & 0 & * & *   
\end{pmatrix}\right\},\;\; \qt'=\left\{ 
\begin{pmatrix}
 * & * & * & * & * & *\\
  * & * & * & * & * & *\\
   0 & 0 & * & * & * & *\\
    0 & 0 & * & * & * & *\\
   0& 0 & 0 & 0 & * & *\\
   0& 0 & 0 & 0 & * & *   
\end{pmatrix}\right\}, \\
\qt&+\bar{\qt}=\qt'+\bar{\qt}'=\left\{ 
\begin{pmatrix}
 * & * & * & * & * & *\\
  * & * & * & * & * & *\\
   0 & * & * & * & * & *\\
    0 & * & * & * & * & *\\
   * & *& * & * & * & *\\
   0& * & 0 & 0 & * & * 
\end{pmatrix}
\right\}.
\end{align*}
 Then $(\gt_0,\qt)$ is not weakly nondegenerate, because $\qt\subsetneqq\qt'\subset\qt+\bar{\qt},$ while
 $(\gt_0,\breve{\qt}_0,\Tilde{\qt}_0)$ is a contact triple which is 
 nondegenerate because $\su(1,5)$ is simple and
 therefore does not contain nontrivial ideals.
\end{exam}

\section{Homogeneous contact structures} \label{hcs}
Let $\Gf_0$ be a Lie group, acting transitively on a smooth manifold $M_0.$ Fix a point $\pnt_0$ 
of $M_0.$ The injective quotient of 
\begin{equation*}
 \piup:\Gf_0\ni{x}\to x\cdot\pnt_0\in{M_0}
\end{equation*}
yields the idenfication 
$M_0\thickapprox\Gf_0/\Hf_0$ of $M_0$ with the quotient of  $\Gf_0$ by the stabilizer $\Hf_0$
of $\pnt_0.$ A $\Gf_0$-equivariant contact structure on $M_0$ is the datum of a constant
rank distribution $\Lcr^*$ on $M_0,$ which is invariant for the left translations
by elements of $\Gf_0:$ 
\begin{equation*}
 x\cdot\Lcr^*=\Lcr^*,\;\;\forall x\in\Gf_0.
\end{equation*}
The pullback 
${\Lcr}=\piup^*(\Lcr^*)$ is a left-invariant distribution on $\Gf_0$ 
generated by a subspace $\lgt_0$
of left-invariant
vector fields on $\Gf_0,$ containing the Lie algebra $\isot_0$ of $\Hf_0.$ 
Moreover, 
the vector subspace 
\begin{equation*} L_{p_0}=\{\Theta_{p_0}\mid \Theta\in\Lcr^*\}\subseteq{T}_{\pnt_0}M_0
\end{equation*}
must be invariant for the differential at $p_0$
of the translations by elements of $\Hf_0.$ This yields 
\begin{equation*}
 \Ad(x)(\lgt_0)=\lgt_0,\;\;\forall x\in\Hf_0,
\end{equation*}
which also implies that $[\isot_0,\lgt_0]\subseteq\lgt_0$ for the Lie algebra $\isot_0$ of $\Hf_0.$  
\par 
Vice versa, il $\lgt_0$ is an $\Ad(\Hf_0)$-invariant linear
subspace of the Lie algebra $\gt_0$ of $\Gf_0,$ 
 then the push-forward on $M_0$ of the distribution on $\Gf_0$ generated by the
left-invariant vector fields  corresponding to the elements of $\lgt_0$ is a smooth distribution 
$\Lcr^*=\lgt_0^*$
on $M_0,$ which is invariant by the \mbox{$\Gf_0$-translations} on $M_0.$ \par 
\smallskip
Assume now that we do not know a priori that $M_0$ is a homogeneous space, but we are given 
a constant rank distribution
$\Lcr^*$ on  $M_0$ and a Lie algebra 
${\gt}_0^*$ 
of smooth vector fields on $M_0$ which leave $\Lcr^*$ invariant: this means that
$[\gt_0^*,\Lcr^*]\subseteq\Lcr^*.$ 
\par We say that ${\gt}_0^*$ is \emph{transitive} at $\pnt_0$ if 
\begin{equation*}
 \{X^*_{p_0}\in{T}_{\pnt_0}M_0\mid {X}^*\in{\gt}_0^*\} = T_{p_0}M_0.
\end{equation*}
\par 
Let $\gt_0={{\gt}_0^*}^{\mathrm{opp}},$
where the superscript\, \textquotedblleft\textit{opp}\textquotedblright\;
 means that, if we denote by $X$ the element of $\gt_0$ corresponding
to the vector field $X^*$ of $\gt_0^*,$ then 
\begin{equation*}
 [X,Y]^*=-[X^*,Y^*], \;\;\forall X,Y\in\gt_0.
\end{equation*}
With $L_{\pnt_0}=\{\Theta_{\pnt_0}\mid \Theta\in\Lcr^*\},$ let us set
\begin{equation*}
 \isot_0=\{X\in\gt_0\mid X^*_{\pnt_0}=0\},\quad\lgt_0=\{X\in\gt_0\mid X^*_{\pnt_0}\in{L}_{\pnt_0}\}.
\end{equation*}
\begin{prop}
 If ${\gt}_0^*$ is transitive, then $(\gt_0,\isot_0,\lgt_0)$ is a transitive contact triple. 
\end{prop} 
\begin{proof}
The quotient $\gt_0/\isot_0$ maps injectively into $T_{\pnt_0}M_0$ and therefore is finite dimensional. 
Let $X\in\lgt_0$  and $Y\in\isot_0.$ Then we can find a vector field $\Theta,$
vanishing at $\pnt_0,$ such that $X^*+\Theta\in\lgt^*_0.$ Then 
\begin{align*}
 [Y^*,X^*+\Theta]=[X,Y]^*+[Y^*,\Theta]\in\lgt_0^*.
\end{align*}
Since $[Y^*,\Theta]$ vanishes at $\pnt_0,$ this means that $[X,Y]^*_{\pnt_0}\in{L}_{p_0},$ showing
that $[Y,X]\in\lgt_0.$ 
\end{proof}

\section{$\Z$-graded Lie algebras and a Tanaka's
theorem} \label{secgrad}
We will use the following  criterion  (\cite[\S{11}]{Tan70}):
 \begin{prop}[N.Tanaka] Let 
\begin{equation}
 \G={\bigoplus}_{h\geq{-\muup}}\G_h
\end{equation}
be a $\Z$-graded real Lie algebra, with $\dim_{\R}(\G_h)<+\infty$ for all $h\in\Z,$
having finitely many summands with negative index. Assume that $\G$ is transitive: this means that 
\begin{equation*}
 \{\etaup\in\G_h\mid [\etaup,\G_{-1}]=\{0\}\}=\{0\},\;\;\;\text{if $h\geq{0}.$}
\end{equation*}
Then a necessary and sufficient condition for $\G$ to be  finite dimensional is that 
\begin{equation}
 \G'=\{\etaup\in\G\mid [\etaup,\G_h]=\{0\}\;\text{for}\; h\leq{-2}\}
\end{equation}
 be finite dimensional.\qed
 \end{prop}\label{prop1.7g}
 Let us comment on 
 this criterion. 
 In the following, we assume that $\G$ is transitive. \par 
 Clearly, $\G'$ is a $\Z$-graded Lie subalgebra of $\G$ and 
\begin{equation}
 [\G'_h,\G_{-1}]\subseteq\G'_{h-1},\;\;\forall h\in\Z .
\end{equation}
Given real vector spaces $V,W,$ let 
 $\Mf^h(V,W)$ denote the space of $W$-valued $h$-multilinear forms on $V$ 
and $\Symm^h(V,W)$ the subspace consisting of those which are symmetric. \par 
We define a map $\etaup\to\etaup_h$ of $\G$ into $\Mf^h(\G_{-1},\G)$ by associating to each
$\etaup\in\G$  the multilinear form 
\begin{equation*}
 \etaup_h(\xiup_1,\hdots,\xiup_h)=[\etaup,\xiup_1,\hdots,\xiup_h],\;\; \text{for}\;\; 
 \xiup_1,\hdots,\xiup_h\in\G_{-1}.
\end{equation*}
We also consider the 
 alternate $\G_{-2}$-valued bilinear form on $\G_{-1}:$ 
\begin{equation}\label{eq3.6e}
 \omegaup(\xiup_1,\xiup_2)=[\xiup_1,\xiup_2],\;\; \text{for}\;\; \xiup_1,\xiup_2\in\G_{-1}
\end{equation}
and set
 \begin{equation}
 \spt(\omegaup)=\{T\in\gl_{\R}(\G_{-1})\mid \omegaup(T(\xiup_1),\xiup_2)+\omegaup(\xiup_1,T(\xiup_2))=0\}.
\end{equation} 

\begin{lem}\label{lem3.2}
 For each $h\geq{0},$ the maps 
\begin{equation} \label{eq3.7c}
\begin{cases}
\G_0\ni \etaup \longrightarrow \etaup_{1}\in\gl_{\R}(\G_{-1}), &\text{for $h=0,$}\\
 \G_h\ni\etaup\longrightarrow \etaup_h\in\Mf^h(\G_{-1},\G_0), &\text{for $h>0,$} 
 \end{cases}
\end{equation}
are injective. 
\begin{equation}\label{eq3.9c}
\begin{cases}
 \G_0' %\subseteq\{\etaup\mid [\etaup,\G_{-2}]=\{0\}\}
 \subseteq \{\etaup\in\G_0\mid \etaup_{1}\in\spt(\omegaup)\}, &
 \text{for $h=0,$}
 \\
\G'_h\subseteq
% \{\etaup\in\G_h\mid [\etaup,\G_{-2}]=\{0\}\}\subseteq
 \{\etaup\in\G_h\mid \etaup_k\in\Symm^{k}(\G_{-1},\G_{h-k})\}, &
 % \{\etaup\in\G_h\mid \etaup_h\in \Symm^h(\G_{-1},\G'_0)\}, 
 \text{for $h,{k}>{0}.$}
\end{cases}
\end{equation}
\end{lem} 
\begin{proof}
 The fact that the maps in \eqref{eq3.7c} are injective is a straightforward consequence of transitivity. \par 
If $\etaup\in\G_0$  and  $[\etaup,\G_{-2}]=0,$ then 
\begin{equation*}
 0=[\etaup,[\xiup_1,\xiup_2]]=[[\etaup,\xiup_1],\xiup_2]+[\xiup_1,[\etaup,\xiup_2]]=\omegaup(\etaup_{1}(\xiup_1),
 \etaup_2)+\omegaup(\xiup_1,\etaup_1(\xiup_2))
\end{equation*}
shows that  $\etaup_1\in\spt(\omegaup).$ 
In the same way, if $\etaup\in\G$ and $[\etaup,\G_{-2}]=0,$ 
 we obtain that 
\begin{equation*}
 [\etaup,\xiup_1,\xiup_2]=[[\etaup,\xiup_1],\xiup_2]=[[\etaup,\xiup_2],\xiup_1]=[\etaup,\xiup_2,\xiup_1],\;
 \forall \etaup\in\G',\; \forall \xiup_1,\xiup_2\in\G_{-1}.
\end{equation*}
% This yields the inclusions \eqref{eq3.9c}.
This gives at once that $[\etaup,\xiup_1,\xiup_2,\xiup_3]=[\etaup,\xiup_2,\xiup_1,\xiup_3]$ for
$\etaup\in\G'$ and $\xiup_1,\xiup_2,\xiup_3\in\G_{-1},$ while 
\begin{align*}
0= [\etaup,[\xiup_1,\xiup_2,\xiup_3]]=[\etaup,[[\xiup_1,\xiup_2],\xiup_3]]=[[\etaup,[\xiup_1,\xiup_2]],\xiup_3]
 +[[\xiup_1,\xiup_2],[\etaup,\xiup_3]]\\
 = [[\xiup_1,[\etaup,\xiup_3]],\xiup_2]+[\xiup_1,[\xiup_2,[\etaup,\xiup_3]]]=-[\etaup,\xiup_3,\xiup_1,\xiup_2]
 +[\etaup,\xiup_3,\xiup_2,\xiup_1]
\end{align*}
shows that also 
$[\etaup,\xiup_1,\xiup_2,\xiup_3]=[\etaup,\xiup_1,\xiup_3,\xiup_2]$ for
$\etaup\in\G'$ and $\xiup_1,\xiup_2,\xiup_3\in\G_{-1}.$
This yields symmetry on the triples. Arguing recursively on $k,$ we obtain, for all ${k}>3,$  
\begin{align*}
0&=[[\etaup,\xiup_1,\hdots,\xiup_{k-2}],[\xiup_{k-1},\xiup_k]]\\
& =[\etaup,\xiup_1,\hdots,\xiup_{k-1},\xiup_k]-
[\etaup,\xiup_1,\hdots,\xiup_{k},\xiup_{k-1}].
\end{align*}
This shows that $[\etaup,\xiup_1,\hdots,\xiup_{k-1},\xiup_k]=
[\etaup,\xiup_1,\hdots,\xiup_{k},\xiup_{k-1}].$ Thus $\etaup_k$ stays invariant under the transposition
$(k-1,k).$ By the recursive assumption, it is also invariant under the transpositions $(j-1,j)$ for
$2\leq{j}<k$ and thus is invariant under the full permutation group of $\{1,\hdots,k\}.$ 
\end{proof}
\begin{esemp}
 Let $V$ be a real vector space of finite dimension $n,$ viewed as a degree $(-1)$-homogeneous
  Abelian real Lie algebra. Then its maximal
 Levi-Tanaka extension is isomorphic to 
 the $\Z$-graded 
 Lie algebra $\Pp$ of vector fields with polynomial coefficients in $\R^n,$ the grading being defined by
\begin{equation}
 \Pp_h=\left.\left\{{\sum}_{i=1}^n p_i(x)\frac{\partial}{\partial{x}_i}\right| p_i\in\R_{h+1}[x_1,\hdots,x_n]
 \right\},\;\;\text{$h\geq{-1}.$}
\end{equation}
Here $\R_{h+1}[x_1,\hdots,x_n]$ denotes the vector space of homogeneous polynomials of degree $(h+1)$
in the $x_1,\hdots,x_n$ variables. \par
If $n=2m$ and $V$ has a complex structure, then the Levi-Tanaka extension of $V\oplus\gl_{\C}(V)$ 
is isomorphic to the $\Z$-graded complex Lie algebra $\Pp^{\C},$ of homolorphic complex vector fields
with holomorphic polynomial coefficients, with the gradation defined by
\begin{equation}
  \Pp_h^{\C}=\left.\left\{{\sum}_{i=1}^m p_i(z)\frac{\partial}{\partial{z}_i}\right| p_i\in\C_{h+1}[z_1,\hdots,z_n]
 \right\},\;\;\text{$h\geq{-1}.$}
\end{equation}
 Here $\C_{h+1}[z_1,\hdots,z_n]$ denotes the vector space of homogeneous 
 holomorphic polynomials of degree $(h+1)$
in the $z_1,\hdots,z_n$ variables. 
\end{esemp} 
\begin{proof}
 The fact that $\Pp$ is a maximal transitive extension of $\Pp_{-1}\simeq{V}$ 
is a consequence of the fact that
 $[\Pp_h,\Pp_{-1}]=\Pp_{h-1}$ for $h\geq{0}.$ \par 
 Analogously, when $V$ has a complex structure,  
$\Pp^{\C}$ is a maximal transitive extensions of  $V\oplus\gl_{\C}(V)\simeq\Pp^{\C}_{-1}\oplus\Pp^{\C}_0$
because
$[\Pp_h^{\C},\Pp_{-1}^{\C}]=\Pp_{h-1}^{\C}$ for all $h\geq{0}.$ 
\end{proof} 

\section{A finitness criterion for $CR$ algebras}\label{finitecr}
We recall that the contact triple associated to  a fundamental $CR$ algebra $(\gt_0,\qt)$  is
$$ (\gt_0,\breve{\qt}_0,\Tilde{\qt}_0)=
(\gt_0,\qt\cap\bar{\qt}\cap\gt_0,(\qt+\bar{\qt})\cap\gt_0).$$
We consider the canonical filtration \eqref{eq2.2a} 
of the contact pair $(\gt_0,\Tilde{\qt}_0)$
and the corresponding $\Z$-graded Lie algebra
\begin{equation}
 \G={\bigoplus}_{h\in\Z}\G_h,\;\;\text{with}\; \; \G_h=\Go_h/\Go_{h+1}.
\end{equation}
Denote by
$\piup_h:\Go_h\to\G_h=\Go_h/\Go_{h+1}$  the projections onto the quotients.
\begin{lem}[Partial complex structure]
There is a unique complex structure $J$ on $\G_{-1},$ defined by 
\begin{equation}
\Jd(\piup_{-1}(X))=\piup_{-1}(Y)\;\; \text{iff \, $X+iY\in\qt.$}\end{equation}
The operator $\Jd\in\gl_{\R}(\G_{-1})$ satisfies 
\begin{equation}\label{eq4.5a}
 [\Jd(\xiup_1),\Jd(\xiup_2)]=[\xiup_1,\xiup_2],\;\; [\Jd(\xiup_1),\xiup_2]+[\xiup_1,\Jd(\xiup_2)]=0,\;\;
 \forall \xiup_1,\xiup_2\in\G_{-1}.
\end{equation}
\end{lem} 
\begin{proof}
To show that $\Jd$ is well defined, we need to verify that, if $X,Y\in\gt_0,$  $X\in\Go_0$ and
$X+\iq{Y}\in\qt,$ then $Y\in\Go_0.$ If $X'\in\Go_{-1},$ then we can find $Y'\in\Go_{-1}$ such that
$X'+\iq{Y}'\in\qt.$ Then 
\begin{align}\label{eq4.6}
 [X+\iq{Y},X'+\iq{Y}']=[X,X']-[Y,Y']+\iq([X,Y']+[Y,X'])\in\qt.
\end{align}
Since by assumption both $[X,X']$ and $[X,Y']$ belong to $\Go_{-1},$ then also $[Y,X']$
and $[Y,Y']$ belong to
$\Go_{-1}.$ This shows that $[Y,\Go_{-1}]\subset\Go_{-1},$ and then $Y\in\Go_0.$ 
Formula \eqref{eq4.6} holds in general for $X+\iq{Y},X'+\iq{Y'}\in\qt,$ yielding \eqref{eq4.5a}. 
\end{proof}
\begin{lem}
 The $\G_{-2}$-valued 
 form \eqref{eq3.6e} is nondegenerate, i.e. 
\begin{equation*}\vspace{-19pt}
 \xiup\in\G_{-1},\; \; \omegaup(\xiup,\xiup')=0,\;\forall \xiup'\in\G_{-1} \Longleftrightarrow
 \xiup=0.
\end{equation*}\qed
\end{lem}
\begin{lem} Let $(\gt_0,\qt)$ be a $CR$ algebra and \eqref{eq2.2a} the associated $\Z$-filtration.
assume that there is a   nonnegative integer $k$ such that $\Go_k\subset\isot_0=\qt\cap\bar{\qt}\cap\gt_0.$ 
Then $\G'_{2k+1}=\{0\}.$ 
\end{lem} 
\begin{proof} By the assumption, any
$Y$ of $\Go_k$ belongs to $\qt$ and therefore 
\begin{equation*}
 [Y,Z_1,\hdots,Z_{k+2}]\in\qt ,\;\; \forall Z_1,\hdots,Z_{k+2}\in\qt.
\end{equation*}
In the complexification of $\G,$ this yelds the equation 
\begin{equation*}
 [\thetaup,\xiup_1+\iq\Jd(\xiup_1),\hdots,
 \xiup_{k+2}+\iq\Jd(\xiup_{k+2})]=0,\;\; \forall\thetaup\in\G_k,\;\;
 \forall \xiup_1,\hdots,\xiup_{k+2}
\in\G_{-1},
\end{equation*}
as $\piup_{-2}(\qt)=\{0\},$ because $\qt$ is contained in the complexification of $\Go_{-1}.$ \par 
Let now $\etaup\in\G_{2k+1}.$ Fix $\xiup\in\G_{-1}$ and consider,
for $0\leq{h}\leq{2k+3},$ 
 the $(2k+4)$ vectors of $\G_{-2}:$ 
\begin{equation*} 
\vq_h=[\etaup,\xiup_1,\hdots,\xiup_{2k+3}],\; \text{with $\xiup_i=\Jd(\xiup)$\, for $i\leq{h}$ and
$\xiup_i=\xiup$\, for $h<i\leq{2k+3}.$}
\end{equation*}
For each choice of $\xiup_1,\hdots,\xiup_{k+1}$ in $\G_{-1},$ the multi-commutator
$[\etaup,\xiup_1,\hdots,\xiup_{k+1}]$ belongs to $\G_k$. Thus, setting, for each
integer $h$ with $0\leq{h}\leq{k+1},$
$\xiup_i^h=\Jd(\xiup)$ if $1\leq{i}\leq{h}$ and $\xiup_i^h=\xiup$ if $h<i\leq{k+1},$ the real 
and imaginary parts of 
\begin{equation*}
 [\etaup,\xiup^h_0,\hdots,\xiup^h_{k},Z_1,\hdots,Z_{k+2}]=0,\;\; \text{with\; $Z_1=\cdots=Z_{k+2}=\xiup+i\Jd(\xiup),$
}
\end{equation*}
yield $(2k+4)$ linear combinations of  $\vq_0,\hdots,\vq_{2k+3}$ 
which sum to zero.
These can be written in the form 
\begin{equation}\label{eq4.7}
 (\vq_0,\vq_1,\hdots,\vq_{2k+3}) \cdot M_k =0,
\end{equation}
where $M_k$ is a real $(2k+4)\times(2k+4)$ 
matrix $M_k$ 
whose columns are the coefficients of the real and imaginary 
parts of the polynomials $t^h(t-i)^{k+2},$ for
$0\leq{t}\leq({k+1}).$ \par
It is easy to check that these polynomials
form a basis of the $(2k+4)$-dimensional
$\C$-vector space $\C_{2k+3}[t]$ of polynomials of degree less
or equal to $(2k+3).$ \par \smallskip
{\small
In fact their $\C$-linear span contains
the $(2k+4)$  polynomials $(t-i)^{k+h+2}$ and
$(t+i)^{k+h+2},$ for $0\leq{h}\leq{k+1}.$ 
These are linearly independent and hence form a basis of $\C_{2k+3}[t].$
Indeed, let $a_h,b_h$ be complex coefficients for which
\begin{equation*}
 f(t)={\sum}_{h=0}^{k+1}(a_h(t-i)^{k+h+2}+b_h(t+i)^{k+h+2})=0.
\end{equation*}
 Let $r\leq{k+1}$ and assume that we already know that $a_h=0$ and $b_h=0$ if
$h>r,$ this being obviously the case when $r=(k+1).$
 Then 
\begin{equation*}
0= \frac{i}{2\cdot(2k+r)!}\frac{d^{2k+r-1}f(-i)}{dt^{2k+r-1}}=a_r
,\quad 0= \frac{-i}{2\cdot(2k+r)!}
\frac{d^{2k+r-1}f(i)}{dt^{2k+r-1}}= b_{r}
\end{equation*}
shows that also $a_r=0$ and $b_r=0.$ By recurrence, this proves that all coefficients $a_h,b_h$ must be zero
and thus the claimed linear independence of the polynomials $(t-i)^{k+h+2},$ 
$(t+i)^{k+h+2}.$}\par\smallskip
Hence $M_k$ is 
 nondegenerate
and \eqref{eq4.7} tells us that all vectors
$\vq_0,\hdots,\vq_{2k+3}$ are zero. 
In particular $\vq_0=0$ and therefore we proved that
 $$[\etaup,\underset{\text{$(2k+3)$ times}}{\underbrace{\xiup,\hdots\hdots,\xiup}}]=0,\;\;
\forall\xiup\in\G_{-1}.$$
Since, by Lemma~\ref{lem3.2}, the multilinear $\G_{-2}$-valued form 
 $\etaup_{2k+3}$ is symmetric in its arguments, it follows by polarization (see e.g. \cite[p.5]{HW39})
 that \par\centerline{$[\etaup,\xiup_1,\hdots,\xiup_{2k+3}]=0$ for all $\xiup_1,\hdots,\xiup_{2k+3}\in\G_{-1}.$}
 \par\noindent
 Since $\omegaup$ is nondegenerate, this yields \par\centerline{$[\etaup,\xiup_1,\hdots,\xiup_{2k+2}]=0$ 
 for all $\xiup_1,\hdots,\xiup_{2k+2}\in\G_{-1}$}\par\noindent and hence, by transitivity,  $\etaup=0.$ 
\end{proof}
Thus we obtain, using also \cite[Theorem 10.2]{MN05},
\begin{thm}\label{thm6.4a}
A $CR$ algebra $(\gt_0,\qt)$ for which the associated contact triple 
$(\gt_0,\breve{\qt}_0,\Tilde{\qt}_0)=(\gt_0,\qt\cap\bar{\qt}\cap\gt_0,(\qt+\bar{\qt})\cap\gt_0)$ 
is nondegenerate is finite dimensional. \qed
\end{thm}
We will prove in \S\ref{ext} that, under the assumptions of Theorem~\ref{thm6.4a},
$(\gt_0,\qt)$ has a \textit{maximal extension} and
that this is finite dimensional. To this aim we will generalise,
in \S\ref{sectrans},  the construction of \S\ref{hcs}, by
a procedure similar to that of \cite{Conn84, GS}.

\section{Transitive pairs and generalised contact distributions}  \label{sectrans}
\subsection{Vector fields with formal power series coefficients}\label{subs7.1}
Let $V$ be a finite dimensional vector space.
The space $\Fi$ of formal power series associated to $V$ is the 
infinite direct sum 
\begin{equation}\label{eq7.1}
 \Fi={\sum}_{h\geq{0}}\Symm^h(V)=\left.\left\{{\sum}_{h=0}^\infty\alphaup_h\right| \alphaup_h\in\Symm^h(V)\;\forall{h}\right\},
\end{equation}
where $\Symm^h(V)=\Symm^h(V,\R)$ are the real-valued homogeneous multilinear symmetric forms 
of degree $h$ (cf. \S~\ref{secgrad}).
The coefficient $\alphaup_0$ is the \textit{value at $0$} of
${\sum}_{h=0}^\infty\alphaup_h.$ With the standard operations,
$\Fi$ is a local ring, whose
maximal ideal~$\Fi_0$ consists  of formal power series vanishing at
$0.$  
\par 
Each vector $\vq$  of $V$ defines a derivation $D_\vq$ on $\Fi,$ 
whose action  on each  summand $\Symm^h(V)$ is described by 
\begin{equation}
 \left\{ 
\begin{aligned}
 & D_\vq\alphaup =0 \;\; \text{if $\alphaup\in\Symm^0(V)\simeq\R,$}\\
& D_\vq\alphaup = h\cdot \vq\rfloor\alphaup
\in\Symm^{h-1}(V), \;\; \text{if $\alphaup\in\Symm^h(V),$ for $h>0,$}\\
&\text{i.e. $(D_\vq\alphaup)(\vq_1,\hdots,\vq_{h-1})=h\cdot\alphaup(\vq_1,\hdots,\vq_{h-1},\vq),$ $\forall
 \vq_1,\hdots,\vq_{h-1}\in{V}.$}
\end{aligned}\right.
\end{equation} 
The set $\Der(\Fi)$ 
of  derivations of $\Fi$ is the left $\Fi$-module $\Fi\otimes_{\R}{V}$ 
generated by $V.$  Thus any derivation $X^*$ is a formal series 
\begin{equation}\label{eq7.3}
 X^*={\sum}_{h=0}^\infty X^*_h,\;\;\text{with}\;\; X^*_h\in
 \Symm^h(V,V)=\Symm^h(V)\otimes_{\R}V.
\end{equation}
Denote by $\Der_0(\Fi)$ the Lie subalgebra $\Fi_0\otimes{V}$
of derivations vanishing at $0.$ 
\subsection{The case of Lie groups}\label{sub7.2}
Let $\Gf_0$ be a real Lie group. We recall that the left and right invariant vector fields on $\Gf_0$
coincide with the infinitesimal generators $L_X$ and $R_X$ of the  one-parameter groups 
%Left and right invariant vector fields on $\Gf_0$ are the infinitesimal generators 
%$L_X,$ $R_X,$ of
%the one-parameter groups 
\begin{align*}
& \R\times\Gf_0\ni (t,x)\to{x}\cdot\exp(tX)\in\Gf_0,
\;\;\;
 \R\times\Gf_0\ni(t,x)\to\exp(tX)\cdot{x}\in\Gf_0,
\end{align*} 
of
diffeomorphisms of $\Gf_0,$ 
respectively. We have the commutation rules 
\begin{equation} 
\begin{cases}
 [L_X(\vq),L_Y(\vq)]=L_{[X,Y]}(\vq),\\
 [R_X(\vq),R_Y(\vq)]=-R_{[X,Y]}(\vq),\\
 [L_X(\vq),R_Y(\vq)]=0,
\end{cases}\quad \forall X,Y\in\gt_0.
\end{equation} 
\par 
The exponential map is a diffeomorphism of an open neighbourhood
of $0$ in 
its Lie algebra $\gt_0$ onto an open neighborhood of
the identity, % in $\Gf_0.$ 
defining a local chart. In order to determine the Taylor series expansions of
$L_X,$ $R_X$ in %the $\gt_0$-coordinates,
in these coordinates,
it is convenient to consider the identities 
\begin{align*}
 \exp(tX)\cdot\exp(\vq)=\exp(\vq+t\cdot{R}_X(\vq)+0(t^2)),\\
 \exp(\vq)\cdot\exp(tX)=\exp(\vq+t\cdot{L}_X(\vq)+0(t^2)),
\end{align*}
for $\vq,X\in\gt_0.$
Since the maps $X\to{L}_X$ and $X\to{R}_X$ are linear, we obtain
{\small
\begin{align*}
 \exp(tX)\cdot\exp(\vq)=\left(\exp(-\vq)\cdot\exp(-tX)\right)^{-1}=\left(\exp(-\vq-t\cdot{L}_X(-\vq)+0(t^2)\right)^{-1}\\
 =\exp(\vq+t\cdot{L}_X(-\vq)+0(t^2))
\end{align*}}
%{\small
%\begin{align*}
% \exp(\vq)\exp(tX)=(\exp(-tX)\exp(-\vq))^{-1}=(\exp(-\vq+t\cdot{R}_{-X}(-\vq)+0(t^2)))^{-1}\\
% =
% \exp(\vq+t\cdot{R}_X(-\vq)+0(t^2))
%\end{align*}}
showing that
\begin{equation*}
 R_X(\vq)=L_X(-\vq).
\end{equation*}
We can obtain 
the Taylor series expansions of 
$L_X$ and $R_X$ in the $\gt_0$-coordinates   
from the Baker-Campbell-Hausdorff formula  (see e.g. \cite{Dyn47}):
\begin{equation}\label{eq7.4a}
\left\{
\begin{aligned}
 L_X(\vq)&={\sum}_{h=0}^\infty b_h (\ad(\vq))^h(X),\\
 R_X(\vq)&={\sum}_{h=0}^\infty (-1)^h b_h(\ad(\vq))^h(X),
\end{aligned}\right.
\end{equation}
where the coefficients $b_h$ are %$B_h(1)/h!,$ where the $B_h(1)$ are Bernoulli numbers, 
defined by 
\begin{equation*}
{\sum}_{h=0}^\infty b_h{t}^h= \frac{t}{1-\exp(-t)}
=1+\frac{t}{2}+\frac{t^2}{12}+\frac{t^3}{48}+\cdots 
\end{equation*}
%\begin{align*} \frac{t}{1-\exp(-t)}
%={\sum}_{h=0}^\infty b_h{t}^h,\; \;\;\text{i.e.}\;\; \;
% b_h=\frac{1}{h!}{\sum}_{r=0}^h{\sum}_{j=0}^r\frac{(-1)^j}{r+1}\binom{r}{j}(1+j)^h,\\
% \intertext{or, recursively, by} b_0=1,\;\; {b}_{h+1}=
% \frac{1}{(h+1)!}-{\sum}_{j=0}^h\frac{b_j}{j!\cdot(h+2-j)!}.\qquad\qquad
%\end{align*}
\par
%We have the commutation rules 
%\begin{equation} 
%\begin{cases}
% [L_X(\vq),L_Y(\vq)]=L_{[X,Y]}(\vq),\\
% [R_X(\vq),R_Y(\vq)]=-R_{[X,Y]}(\vq),\\
% [L_X(\vq),R_Y(\vq)]=0,
%\end{cases}\quad \forall X,Y\in\gt_0.
%\end{equation}
\subsection{Homogeneous spaces}\label{sub8.3}
Let $M=\Gf_0/\Hf_0$ be a homogeneous space, with base point 
$\pnt_0=[\Hf_0].$ 
Fix a
 linear complement $V$ of $\isot_0$ in $\gt_0.$ 
The map 
\begin{equation}
 V\times\Hf_0\ni(\vq,x)\longrightarrow \exp(\vq)\cdot{x}\in\Gf_0
\end{equation}
restricts to a diffeomorphism of the product of a neighborhood of $\{0\}\times\Hf_0$  onto an
open neighbourhood of the identity in $\Gf_0.$ We use the projection on $M$ of $\exp(\vq)$ to define
coordinates near $\pnt_0.$ 
In fact, if $\varpi:\Gf_0\to{M}$ is the natural projection, with $\varpi(x)=x\cdot\pnt_0,$ 
then $\vq\to\varpi(\exp(\vq))$ is a diffeomorphism of an open neighbourhood of $0$ in $V$ onto an
open neighbourhood of $\pnt_0$ in $M.$ 
Let $\Fi$ be as in \S\ref{subs7.1}.
\par
In analogy with the definitions of % the vector fields 
$L_X$ and $R_X$ in \S\ref{sub7.2}, 
we may introduce a couple of vector fields $L^*_X$ 
and  $R_X^*$ on $M, $
as infinitesimal generators of the one-parameter groups of diffeomorphisms of $M,$ locally defined by
%Analogues of $L_X$ and $R_X$ in the $V$ patch in $M$ are vector fields in $\Der(\Fi),$
% that 
%will be denoted by $L_X^*$ and
%$X^*,$ respectively. They are described by
{\small
\begin{equation}\label{eq7.7b}
\begin{cases}
\exp(\vq)\cdot\exp(tX)=\exp(\vq+t\cdot{L}^*_X(\vq)+0(t^2))\cdot
\exp(t\cdot{H}(\vq)+0(t^2)),\\
\exp(tX)\cdot\exp(\vq)=\exp(\vq+t\cdot{R}_X^*(\vq)+0(t^2))\cdot
\exp(t\cdot{H}'(\vq)+0(t^2)),\\
\text{with}\;\; L_X^*,R_X^*\in\Fi\otimes{V},\;\; H,H'\in\Fi\otimes\isot_0.
\end{cases}
\end{equation}}
Let us find their formal power series expansions. 
Set 
\begin{equation}\label{eq7.8c}\left\{
\begin{aligned}
R_X^*(\vq)={\sum}_{h=0}^\infty\xq_h(\vq),\quad
H'(\vq)={\sum}_{h=0}^\infty\hq'_h(\vq),\;\;\;\qquad\qquad \\
\text{with\;\;
$\xq_h\in\Symm^h(V,V),$ \; $\hq'_h\in\Symm^h(V,\isot_0).$}
\end{aligned}\right. \end{equation}
Using \eqref{eq7.4a}, we obtain 
\begin{align*}
{\sum}_{h=0}^{\infty}(-1)^h{b}_h(\ad(\vq))^h(X)={\sum}_{h=0}^\infty
\xq_h(\vq)+{\sum}_{h=0}^\infty{\sum}_{r+s=h}b_r(\ad(\vq))^r(\hq'_s
(\vq)),
\end{align*}
yielding by recurrence 
\begin{equation}\label{eq7.8a}
\begin{cases}
X = \xq_0+\hq'_0,\\ \begin{aligned}
\xq_{h+1}(\vq)+\hq'_{h+1}(\vq)=(-1)^{h+1}b_{h+1}(\ad(\vq))^{h+1}(X)
\qquad\quad  \\
-{\sum}_{r=0}^{h}{b}_{r+1}(\ad(\vq))^{r+1}(\hq'_{h-r}(\vq)),\;\;
\end{aligned}\quad (h\geq{0}).
\end{cases}
\end{equation}
Let $\piup:\gt_0\to{V}$ be the projection along $\isot_0.$
Equations \eqref{eq7.8a} can be used to obtain explicit formulae
for $\xq_h$ and $\hq'_h$:
\begin{equation}\label{eq7.9b}
\begin{cases}
\xq_0=\piup(X),\\
\hq_0'=X-\piup(X),\\
\begin{aligned}
\xq_{h+1}(\vq)=(-1)^{h+1}b_{h+1}\piup((\ad(\vq))^{h+1}(X))\qquad\qquad\qquad
\\ -{\sum}_{r=0}^h
b_{r+1}\piup((\ad(\vq))^{r+1}(\hq'_{h-r}(\vq))),
\end{aligned}
%\quad (h\geq{0}),
\\[22pt]
\begin{aligned}
\hq'_{h+1}(\vq)=(-1)^{h+1}b_{h+1}(\ad(\vq))^{h+1}(X)\qquad\qquad\qquad
\qquad\quad
\\
-{\sum}_{r=0}^{h}{b}_{r+1}(\ad(\vq))^{r+1}(\hq'_{h-r}(\vq))
-\xq_{h+1}(\vq).\end{aligned} %\quad (h\geq{0}). 
\end{cases}
\end{equation}
Since the projection $\Gf_0\ni{x}\to{x}\cdot\pnt_0\in{M}$ relates 
$R_X^*$ with the right invariant vector field on $\Gf_0$ corresponding to $X,$ 
we obtain:
\begin{lem} \label{lem7.1}
We have
\begin{equation}
\vspace{-21pt}
[X^*(\vq),Y^*(\vq)]=-[X,Y]^*(\vq),\;\;\forall X,Y\in\gt_0.
\end{equation}\qed 
\end{lem} 
Analogously, let us set 
\begin{equation}\label{eq7.11c}\left\{
\begin{aligned}
L_X^*(\vq)={\sum}_{h=0}^\infty\yq_h(\vq),\quad
H(\vq)={\sum}_{h=0}^\infty\hq_h(\vq),\;\;\;\qquad\qquad \\
\text{with\;\;
$\yq_h\in\Symm^h(V,V),$ \; $\hq_h\in\Symm^h(V,\isot_0).$}
\end{aligned}\right.\end{equation}
From the equation 
\begin{align*}
{\sum}_{h=0}^{\infty}{b}_h(\ad(\vq))^h(X)={\sum}_{h=0}^\infty
\yq_h(\vq)+{\sum}_{h=0}^\infty{\sum}_{r+s=h}b_r(\ad(\vq))^r(\hq_s
(\vq)),
\end{align*}
we obtain 
{%\small
\begin{equation}\label{eq7.11a}
\begin{cases}
X = \yq_0+\hq_0,\\ \begin{aligned}
\yq_{h+1}(\vq)+\hq_{h+1}(\vq)=b_{h+1}(\ad(\vq))^{h+1}(X)
\qquad
\qquad\quad \quad\\
-{\sum}_{r=0}^{h}{b}_{r+1}(\ad(\vq))^{r+1}(\hq_{h-r}(\vq)),\;\;\\
\end{aligned}(h\geq{0}).
\end{cases}
\end{equation}}
Equations \eqref{eq7.11a} can be used to obtain recursive formulae
for $\yq_h$ and $\hq_h$:
\begin{equation}\label{eq7.12b}
\begin{cases}
\yq_0=\piup(X),\\
\hq_0=X-\piup(X),\\
\begin{aligned}
\yq_{h+1}(\vq)=b_{h+1}\piup((\ad(\vq))^{h+1}(X))\qquad\qquad\qquad
\\ -{\sum}_{r=0}^h
b_{r+1}\piup((\ad(\vq))^{r+1}(\hq_{h-r}(\vq))),
\end{aligned}
%\quad (h\geq{0}),
\\[22pt]
\begin{aligned}
\hq_{h+1}(\vq)=b_{h+1}(\ad(\vq))^{h+1}(X)\qquad\qquad\qquad
\qquad\quad
\\
-{\sum}_{r=0}^{h}{b}_{r+1}(\ad(\vq))^{r+1}(\hq_{h-r}(\vq))
-\yq_{h+1}(\vq),\end{aligned} %\quad (h\geq{0}). 
\end{cases}
\end{equation}
Note that $L_Y^*=0$ when $Y\in\isot_0.$\par  
 To compute the commutator of $R_X^*$ and $L^*_Y$ for a pair $X,Y\in\gt_0$,
we use the infinitesimal description of their flows,
which can be obtained from \eqref{eq7.7b}.
Set $\Phi_X$ for the flow of $R_X^*$ and $\Psi_Y$
for the flow of $L_Y^*.$ Then  
\begin{align*}
 \exp(\vq+t^2[R_X^*,L_Y^*](\vq)+0(t^3))=\exp(\Phi_X(-t)\circ\Psi_Y(-t)\circ\Phi_X(t)\circ\Psi_Y(t))(\vq)).
\end{align*}
By lifting the action to $\Gf_0,$ we have 
\begin{equation*}\begin{cases}
 \exp(\Phi_X(t)(\vq))=\exp(tX)\exp(\vq)\exp(-tH'(\vq)+0(t^2)),\\
  \exp(\Psi_Y(t)(\vq))=\exp(\vq)\exp(tY)\exp(-tH(\vq)+0(t^2)).
  \end{cases}
\end{equation*}
Then we obtain the composition 
{\small
\begin{align*}
 \exp(-tX)
 \exp(tX)
 \exp(\vq)\exp(tY)\exp(-tH(\vq)+0(t^2))
 \exp(-tH'(\vq)+0(t^2)) \\
 \cdot 
 \exp(-tY)\exp(tH(\vq))
 \exp(tH'(\vq))\\
 =  \exp(\vq)\exp(tY)\exp(-tH(\vq))
 \exp(-tH'(\vq)) 
 \exp(-tY)\exp(tH(\vq))
 \exp(tH'(\vq))+0(t^3)\\
 = \exp(\vq)\exp(-t^2[Y,(H(\vq)+H'(\vq))])\exp(t^2[H(\vq),H'(\vq)])+0(t^3).
\end{align*}}
This yields
\begin{lem} \label{lem7.2}
For $X,Y\in\gt_0,$ we have 
\begin{equation}
 [R_X^*,L_Y^*](\vq)=L^*_{[H+H',Y]}(\vq),
\end{equation}
 where $H(\vq)$ and $H'(\vq)$ are described by \eqref{eq7.12b} 
 and \eqref{eq7.9b}, and the right hand side is a composition of formal power series.
 \qed 
\end{lem}
This lemma tells us that the \textit{infinitesimal
translations} of $L^*_{\wq},$
for a $\wq\in{V},$ can be expressed as formal power series 
whose coefficients are $L^*_Y$'s for $Y$ in the $\isot_0$-module
generated by $\wq.$ 
In a similar way we also obtain 
\begin{lem}
 For $Y_1,Y_2\in\gt_0,$ we have 
\begin{equation}
 [L^*_{Y_1},L^*_{Y_2}]=L^*_{[Y_1,Y_2]+[H_1(\vq)-H_2(\vq),Y_1-Y_2]},
\end{equation}
where $H_i\in\Fi\otimes\isot_0$ is defined by the first line of \eqref{eq7.7b}, with $X=Y_i$, for $i=1,2.$
\qed
\end{lem}
\subsection{General transitive pairs} 
Let us fix a transitive pair.
We recall from Definition~\ref{def1.1} that
it is a couple $(\gt_0,\isot_0)$ consisting of a linearly compact topological Lie algebra
$\gt_0$ and of a finite codimensional closed
subalgebra $\isot_0$ that does not contain any nontrivial
ideal of $\gt_0.$  
\par 
Let us fix a finite dimensional complement $V$ of $\isot_0$ in
$\gt_0.$ 
We define $R_X^*(\vq),\, L_X^*(\vq),\, H(\vq),\, H'(\vq)$ by
\eqref{eq7.8c}, \eqref{eq7.8a}, \eqref{eq7.11c}, \eqref{eq7.11a},
after noticing that to
write these formulae it is not needed that
$\gt_0$ be finite dimensional,
% We note that, in the case of homogeneous spaces of \S\ref{sub8.3},
because the homogeneous summands in the $V$-coordinates of 
their Taylor series %describing their homogeneous components 
only involve finite
powers of $\ad(\vq),$ finite 
linear
combinations, and the projection
$\piup:\gt_0\to{V}$ along $\isot_0.$
Set 
\begin{equation}
 \gt_0^*=\{R_X^*\mid X\in\gt_0\}.
\end{equation}
\begin{thm} \label{thm7.4}
If $(\gt_0,\isot_0)$ is a transitive pair and $V$
a linear complement of $\isot_0$ in $\gt_0$, then the map  
\begin{equation}\label{eq7.16}
\gt_0\ni{X}\longrightarrow{R}_X^*\in\Der(\Fi)
\end{equation}
is injective and defines an 
anti-isomorphism of Lie algebras
between $\gt_0$ and $\gt_0^*,$ with 
\begin{equation}
 \isot_0=\{X\in\gt_0\mid R_X^*\in\Der_0(\Fi)\}.
\end{equation}
The correspondence 
\begin{equation}\label{eq7.16b}
\lgt_0 \longrightarrow \lgt_0^*=\Fi\otimes\langle L^*_\wq\mid \wq\in{V}\cap
\lgt_0\rangle
\end{equation}
is a bijection between the set of $\isot_0$-submodules of $\gt_0$
containing $\isot_0$ 
and $\gt_0^*$-invariant left $\Fi$-modules of $\Der(\Fi).$ Its inverse is given by 
\begin{equation}
 \lgt_0^*\longrightarrow \lgt_0=\{Y\in\gt_0\mid \piup(Y)=\Theta(0)\;\;\text{for some}\;\; \Theta\in\lgt_0^*\}.
\end{equation}
\end{thm} \begin{proof}
Let $X\in\gt_0$ and assume that $R_X^*=0.$ We use \eqref{eq7.9b}.
From $\xq_0=0$ se obtain that $X\in\isot_0.$ Since 
$b_1=\tfrac{1}{2}\neq{0},$ the condition that $\xq_1=0$ means that
$[\vq,X]\in\isot_0$ for all $\vq\in{V},$ 
and hence $[Y,X]\in\isot_0$ for all $Y\in\gt_0.$ In general, we find that 
$\hg_h(\vq),$ for each $h\geq{0},$ is 
a multiple of $(\ad(\vq)^h(X)$ and, 
arguing by recurrence we
obtain that 
$(\ad(Y)^h(X)\in\isot_0$ for all $h\geq{0}$ 
and $Y\in\gt_0.$ 
This yields that actually $\ad(Y_1)\circ\cdots \circ\ad(Y_h)(X)
\in\isot_0$
for all $h$ and $Y_1,\hdots,Y_h\in\gt_0.$ To show this fact, we argue again 
 by recurrence on $h,$ as the cases of $h=0,1$ are already
settled. For $h>1,$ we note that, for $t_1,\hdots,t_h\in\R,$  
\begin{align*}
t_1\cdots{t}_h\cdot 
\ad(Y_1)\circ\cdots\circ\ad(Y_h)(X)= 
\frac{1}{h!}[\ad(t_1{Y}_1+\cdots+t_h{Y}_h)]^h(X)
+{\sum}_i0(t_i^2)\\
\text{$+$ 
terms of the form $\ad(Y'_1)\circ\cdots\ad(Y'_r)(X)$
with $r<h.$}
\end{align*}
By the recursive assumption, the coefficient of the monomial $t_1\cdots{t}_h$ 
in the right hand side is an element of $\isot_0$
and hence
 $\ad(Y_1)\circ\cdots\circ\ad(Y_h)(X)\in\isot_0.$ 
We proved that the kernel of the map $X\to{R}_X^*$ is an ideal of $\gt_0$
 contained in $\isot_0.$ Therefore it is $\{0\}$ if $(\gt_0,\isot_0)$
 is a transitive pair. \par 
 The conclusion of Lemma~\ref{lem7.1} only depends on the formal
definition of $R_X^*$ in \eqref{eq7.8c} and \eqref{eq7.8a} 
and therefore is still valid, yielding \eqref{eq7.16}. \par 
Also the validity of 
Lemma~\ref{lem7.2} depends only on the formal definitions of
$R_X^*$ and $L_Y^*$ and therefore shows that, when
$\lgt_0$ contains $\isot_0$ and $[\isot_0,\lgt_0]\subseteq\lgt_0,$
then the left $\Fi$-module $\lgt_0^*$ 
generated by $\{L^*_{\wq}\mid \wq\in{V}
\cap\lgt_0\}$ satisfies $[\gt_0^*,\lgt_0^*]\subseteq
\lgt_0^*.$ Vice versa, if $\lgt_0^*$ is a left $\Fi$ submodule
of $\Der(\Fi)$ with $[\gt_0^*,\lgt_0^*]\subseteq\lgt_0^*,$
then the set $\lgt_0$ 
of $Y\in\gt_0$ such that $\piup(Y)$ is the value
in $0$ of a vector field in $\lgt_0^*$ is a subspace of $\gt_0$
containing $\isot_0$ and satisfying $[\isot_0,\lgt_0]\subseteq\isot_0.$
Indeed $[X,\wq]$ is the value at $0$ of $[R_X^*,L^*_{\wq}].$ 
This yields the correspondence \eqref{eq7.16b}, completing the
proof of the theorem.
\end{proof}
We already noted that the map $Y\to{L}^*_Y$ is linear. In particular, it can be extended by $\C$-linearity
to the case where $Y$ belongs to the complexification $\gt$ of $\gt_0$. Then the second part of the
statement of Theorem~\ref{thm7.4} extends to the case of \textit{complex vector distributions}.
We denote by $\isot\subseteq\gt$ the complexification of $\isot_0.$
\begin{thm}\label{thm7.5}
 The correspondence 
\begin{equation}\label{eq7.16e}
\lgt \longrightarrow \lgt^*=\Fi\otimes\langle L^*_\wq\mid \wq\in{V}\cap
\lgt\rangle
\end{equation}
is a bijection between the set of $\isot$-submodules of $\gt$
containing $\isot$ 
and $\gt_0^*$-invariant left $\C\otimes\Fi$-modules of $\C\otimes\Der(\Fi).$ 
Its inverse is given by 
\begin{equation}\vspace{-21pt}
 \lgt^*\longrightarrow \lgt=\{Y\in\gt\mid \piup(Y)=\Theta(0)\;\;\text{for some}\;\; \Theta\in\lgt^*\}.
\end{equation} \qed
\end{thm}
\section{Extensions} \label{ext}
\begin{defn}
 We say that a contact triple $(\gt_0',\isot_0',\lgt'_0)$ \emph{extends} the  contact triple
 $(\gt_0,\isot_0,\lgt_0)$ if there is an injective homomorphism of real Lie algebras $\phiup:\gt_0\to\gt_0'$
 such that $\phiup(\lgt_0)\subseteq\lgt'_0,$ $\phiup(\isot_0)\subseteq\isot_0'$ and the quotient maps 
 $(\gt_0/\lgt_0)\to(\gt_0'/\lgt_0')$ and $(\gt_0/\isot_0)\to(\gt_0'/\isot_0')$ induced by $\phiup$ are linear isomorphisms. \par
 We say that a $CR$ algebra $(\gt_0',\qt')$ \emph{extends} the $CR$ algebra $(\gt_0,\qt)$ 
 if there is an injective Lie algebras homomorphism $\phiup:\gt_0\to\gt_0',$ whose complexification we still
 denote by $\phiup,$ such that $\phiup(\qt)\subseteq\qt'$ and 
 the induced map on the quotients
$\gt_0/(\qt\cap\bar{\qt} \cap\gt_0)\to \gt_0'/(\qt'\cap\bar{\qt}' \cap\gt_0')$ and
$(\gt/\qt)\to(\gt'/\qt')$ are linear isomorphisms.
\end{defn}\par\smallskip
To deal with extensions, it is convenient to introduce a 
common Lie algebra in which we can embed both a given
Lie algebra and its extension. \par 
\begin{prop}
 Let $(\gt_0,\isot_0,\lgt_0)$ be a contact triple. Then there is a maximal contact triple $(\gt_0',\isot_0',\lgt_0')$
 extending $(\gt_0,\isot_0,\lgt_0),$ which is unique modulo isomorphisms. \par 
  Let $(\gt_0,\qt)$ be a $CR$ algebra. Then there is a maximal $CR$ algebra $(\gt_0',\qt')$
 extending $(\gt_0,\qt),$ which is unique modulo isomorphisms.
\end{prop} 
\begin{proof}
The statement follows from Theorems~\ref{thm7.4} and~\ref{thm7.5}.
If $\lgt_0^*$ is the $\Fi$-module corresponding to $\lgt_0$, we define ${\gt_0'}^*$ as the Lie algebra of
formal vector fields stabilising $\lgt^*_0$ and $\gt_0'$ equals to its opposite Lie algebra. \par
Likewise, in the case of a $CR$ algebra, we take ${\gt_0'}^*$ to be the stabiliser of ${\Tilde{\qt}}^*$ in
$\Der(\Fi)$ and define $\gt_0'$ to be its opposite Lie algebra.
\end{proof}
The finiteness result of \S\ref{finitecr} applies to give informations about
the global and local $CR$ automorphisms  
 on homogeneous and locally homogeneous $CR$ manifolds.
\par 
The analytic 
Lie subgroup of a Lie group $\Gf_0$ generated by a Lie subalgebra
$\isot_0$ of its Lie algebra $\gt_0$ may fail to be closed.
In this case the pair $(\gt_0,\isot_0)$ 
is associated 
to a \emph{locally $\Gf_0$-homogeneous} manifold, i.e.
a smooth open manifold $M_0$ having the property that the elements
of a small open neighborhood of the origin of $\Gf_0$ act 
as a transitive group of local diffeomorphisms on an open neighborhood
of a base point $\pnt_0$ of $M_0,$ and $\isot_0$ is the
Lie algebra of the stabilizer of $\pnt_0$ for this action
(see e.g. \cite{Most1950,Sp1993}). We can give in an obvious way
a notion of \emph{locally $\Gf_0$-homogeneous $CR$ manifold},
that we employ in the formulation of the following result.
\begin{thm}
Every contact nondegenerate $CR$ algebra $(\gt_0,\qt)$ admits an essentially unique maximal extension 
$(\gt_0',\qt'),$ which is finite dimensional and is therefore a $CR$ algebra of a locally homogeneous $CR$ manifold
whose $CR$ automorphisms form a Lie group of transformations.\qed
\end{thm} 
\begin{thm} Let $\Gf_0$ be a Lie group and
$M_0$ a locally $\Gf_0$-homogeneous $CR$ manifold, 
with associated $CR$ algebra $(\gt_0,\qt).$ 
 If $(\gt_0,\qt)$ is fundamental and contact nondegenerate, then the
local  $CR$ automorphisms of $M_0$ generate a  
 finite dimensional Lie group. \qed
\end{thm}
\providecommand{\bysame}{\leavevmode\hbox to3em{\hrulefill}\thinspace}
\providecommand{\MR}{\relax\ifhmode\unskip\space\fi MR }
\providecommand{\MRhref}[2]{%
  \href{http://www.ams.org/mathscinet-getitem?mr=#1}{#2}
}
\providecommand{\href}[2]{#2}

%\bibliographystyle{amsplain}
%\renewcommand{\MR}[1]{}
%\bibliography{homog1-bis}

\begin{thebibliography}{10}

\bibitem{AMN06}
Andrea Altomani, Costantino Medori, and Mauro Nacinovich, \emph{The {CR}
  structure of minimal orbits in complex flag manifolds}, J. Lie Theory
  \textbf{16} (2006), no.~3, 483--530. \MR{2248142}

\bibitem{AMN10b}
\bysame, \emph{On homogeneous and symmetric {CR} manifolds}, Boll. Unione Mat.
  Ital. (9) \textbf{3} (2010), no.~2, 221--265. \MR{2666357}

\bibitem{AMN06b}
\bysame, \emph{Orbits of real forms in complex flag manifolds}, Ann. Sc. Norm.
  Super. Pisa Cl. Sci. (5) \textbf{9} (2010), no.~1, 69--109. \MR{2668874}

\bibitem{AMN2013}
\bysame, \emph{Reductive compact homogeneous {CR} manifolds}, Transform. Groups
  \textbf{18} (2013), no.~2, 289--328. \MR{3055768}

\bibitem{BER1999}
M.~Salah Baouendi, Peter Ebenfelt, and Linda~P. Rothschild, \emph{Real
  submanifolds in complex space and their mappings}, Princeton Mathematical
  Series, vol.~47, Princeton University Press, Princeton, NJ, 1999.
  \MR{1668103}

\bibitem{Conn81}
Jack~F. Conn, \emph{Nonabelian minimal closed ideals of transitive {L}ie
  algebras}, Mathematical Notes, vol.~25, Princeton University Press,
  Princeton, N.J.; University of Tokyo Press, Tokyo, 1981. \MR{595686}

\bibitem{Conn84}
\bysame, \emph{On the structure of real transitive {L}ie algebras}, Trans.
  Amer. Math. Soc. \textbf{286} (1984), no.~1, 1--71. \MR{756031}

\bibitem{Dyn47}
Eugene~B. Dynkin, \emph{Calculation of the coefficients in the
  {C}ampbell-{H}ausdorff formula}, Doklady Akad. Nauk SSSR (N.S.) \textbf{57}
  (1947), 323--326. \MR{0021940}

\bibitem{KaFe2008}
Gregor Fels and Wilhelm Kaup, \emph{Classification of {L}evi degenerate
  homogeneous {CR}-manifolds in dimension 5}, Acta Math. \textbf{201} (2008),
  no.~1, 1--82. \MR{2448066}

\bibitem{Gu1968}
Victor~W. Guillemin, \emph{A {J}ordan-{H}\"older decomposition for a certain
  class of infinite dimensional {L}ie algebras}, J. Differential Geometry
  \textbf{2} (1968), 313--345. \MR{0263882}

\bibitem{GS}
Victor~W. Guillemin and Shlomo Sternberg, \emph{An algebraic model of
  transitive differential geometry}, Bull. Amer. Math. Soc. \textbf{70} (1964),
  16--47. \MR{0170295}

\bibitem{IsZa2013}
Alexander Isaev and Dmitri Zaitsev, \emph{Reduction of five-dimensional
  uniformly {L}evi degenerate {CR} structures to absolute parallelisms}, J.
  Geom. Anal. \textbf{23} (2013), no.~3, 1571--1605. \MR{3078365}

\bibitem{MN97}
Costantino Medori and Mauro Nacinovich, \emph{Levi-{T}anaka algebras and
  homogeneous {CR} manifolds}, Compositio Math. \textbf{109} (1997), no.~2,
  195--250. \MR{1478818}

\bibitem{MN05}
\bysame, \emph{Algebras of infinitesimal {CR} automorphisms}, J. Algebra
  \textbf{287} (2005), no.~1, 234--274. \MR{2134266}

\bibitem{MeSp2014}
Costantino Medori and Andrea Spiro, \emph{The equivalence problem for
  five-dimensional {L}evi degenerate {CR} manifolds}, Int. Math. Res. Not. IMRN
  (2014), no.~20, 5602--5647. \MR{3271183}

\bibitem{MeSp2015}
\bysame, \emph{Structure equations of {Levi} degenerate {CR} hypersurfaces of
  uniform type}, Rend. Semin. Mat. Univ. Politec. Torino (2015), no.~73/1,
  127--150.

\bibitem{Most1950}
George~D. Mostow, \emph{The extensibility of local {L}ie groups of
  transformations and groups on surfaces}, Ann. of Math. (2) \textbf{52}
  (1950), 606--636. \MR{0048464}

\bibitem{Poc2013}
Samuel Pocchiola, \emph{Explicit absolute parallelism for $2$-nondegenerate
  real hypersurfaces ${M}^5\subset \mathbb{C}^3$ of constant {Levi} rank $1$},
  arXiv:1312.6400, 2013.

\bibitem{Po2015}
Curtis Porter, \emph{The local equivalence problem for $7$-dimensional,
  $2$-nondegenerate {CR} manifolds whose cubic form is of conformal unitary
  type}, arXiv:1511.04019, 2015.

\bibitem{PoZe2017}
Curtis Porter and Igor Zelenko, \emph{Absolute parallelism for
  $2$-nondegenerate {CR} structures via bigheaded {Tanaka} prolongation},
  arXiv:1704.03999, 2017.

\bibitem{Sa2015}
Andrea Santi, \emph{Homogeneous models for {Levi} degenerate {CR} manifolds},
  arXiv:1511.08902, 2015.

\bibitem{Sp1993}
Andrea Spiro, \emph{A remark on locally homogeneous {R}iemannian spaces},
  Results Math. \textbf{24} (1993), no.~3-4, 318--325. \MR{1244285}

\bibitem{Tan67}
Noboru Tanaka, \emph{On generalized graded {L}ie algebras and geometric
  structures. {I}}, J. Math. Soc. Japan \textbf{19} (1967), 215--254.
  \MR{0221418}

\bibitem{Tan70}
\bysame, \emph{On differential systems, graded {L}ie algebras and
  pseudogroups}, J. Math. Kyoto Univ. \textbf{10} (1970), 1--82. \MR{0266258}

\bibitem{HW39}
Hermann Weyl, \emph{The {C}lassical {G}roups. {T}heir {I}nvariants and
  {R}epresentations}, Princeton University Press, Princeton, N.J., 1939.
  \MR{0000255}

\end{thebibliography}

\end{document}